\documentclass[11pt,a4paper]{article}
\usepackage[utf8]{inputenc}
\usepackage[T1]{fontenc}
\usepackage{amsmath, amssymb, amsthm, amsfonts,mathrsfs}
\usepackage{float}
\usepackage{geometry}
\usepackage[hidelinks]{hyperref}
\usepackage{tikz-cd}
\usepackage{enumitem}

\newcommand{\Q}{\mathbf{Q}}
\newcommand{\Z}{\mathbf{Z}}
\newcommand{\R}{\mathbf{R}}
\newcommand{\C}{\mathbf{C}}
\newcommand{\N}{\mathbf{N}}
\newcommand{\A}{\mathbb{A}}
\newcommand{\Af}{\mathbb{A}_f}
\def\F{{\mathbf F}}
\def\N{{\mathbf N}}

\def\nat{\widehat{\N}^\times}
\newcommand{\Zhat}{\widehat{\Z}}

\newcommand{\XQ}{X_{\Q}}
\newcommand{\YQ}{Y_{\Q}}
\newcommand{\CQ}{C_{\Q}}

\def\Hom{\rm{Hom}}
\def\spz{\Spec \Z}
\def\spzb{\overline{\Spec \Z}}
\def\znt{{\widehat\Z^\times}}
\newcommand{\Spec}{\operatorname{Spec}}
\def\Pic{\operatorname{Pic}}
\def\Jac{\operatorname{Jac}}
\def\Div{{\rm Div}}
\def\PDiv{{\rm PDiv}}

\def\tspz{\widetilde{\Pic}(\spzb)}
\def\spz{\Spec \Z}

\def\tf1{{\rm torsion free rank one groups}}
\def\MFr{\operatorname{Pic}_{Fr}}
\def\MRt{\operatorname{Pic}_{Rt}}
\def\F1{\mathbf F_1}

\def\Idem{\rm{Idem}}

\def\cS{{\mathcal S}}

\def\Tr{\rm{Trace}}

\def\ie{{\it i.e.\/}\ }

\theoremstyle{plain}
\newtheorem{thm}{Theorem}[section]
\newtheorem{defn}[thm]{Definition}
\newtheorem{prop}[thm]{Proposition}
\newtheorem{fact}[thm]{Fact}

\newtheorem{lem}[thm]{Lemma}
\newtheorem{rem}[thm]{Remark}

\newtheorem*{defn*}{Definition}
\newtheorem*{example*}{Example}
\newtheorem*{prop*}{Proposition}
\newtheorem*{thm*}{Theorem}
\newtheorem*{rem*}{Remark}
\newtheorem*{cor*}{Corollary}
\newtheorem*{lem*}{Lemma}

\theoremstyle{definition} 

\newtheorem*{claim*}{Claim}
\newtheorem*{conj*}{Conjecture}
\newtheorem*{issue*}{Issue}

\title{On the Jacobian of $\spzb$}
\author{Alain Connes and Caterina Consani\thanks{Partially supported by the Travel Support for Mathematicians n. 691493, Simons Foundation.}}
\date{}

\begin{document}

\maketitle

\begin{abstract}
\noindent We interpret the structure of the adele class space of the rationals—and specifically its Riemann sector—as the natural monoidal extension of the Picard group of the arithmetic curve $\overline{\operatorname{Spec} \Z}$.  We identify the elements of this space with torsion-free rank-1 abelian groups $L$ endowed with rigidifying data. In the Riemann sector, this data corresponds to a norm, extending the classical notion of metrized line bundles in Arakelov geometry. For the full adele class space, we replace the norm with a group morphism to $\R$ and a combinatorial datum: a parametrization of the roots of unity associated with the character dual of $L$. We show that the product of adeles is represented geometrically by the tensor product of these rank-1 groups and their rigidifying structures. 
The resulting monoid space generalizes the Picard group to the full adelic context by incorporating the singular strata required for the spectral realization of $L$-functions.


\paragraph{MSC 2020.}  
\href{http://www.ams.org/mathscinet/msc/msc2020.html?t=14G40&btn=Current}{14G40}, 
\href{http://www.ams.org/mathscinet/msc/msc2020.html?t=11M55&btn=Current}{11M55}, 
\href{http://www.ams.org/mathscinet/msc/msc2020.html?t=14C20&btn=Current}{14C20}, 
\href{http://www.ams.org/mathscinet/msc/msc2020.html?t=11R42&btn=Current}{11R42}, 
\href{http://www.ams.org/mathscinet/msc/msc2020.html?t=06F05&btn=Current}{06F05}. 

\end{abstract}
\tableofcontents

\section{Introduction}

The starting point of this paper is the apparent paradox concerning the genus of the arithmetic curve $\spz$. On one hand, the classical Jacobian of $\spz$—its ideal class group—is trivial, suggesting a curve of genus zero. On the other hand, the spectral realization of the zeros of the Riemann zeta function \cite{Co-zeta}, as well as the explicit formulas of Riemann-Weil \cite{Weil}, strongly suggest that $\spz$ should be regarded as a curve of infinite genus.
To resolve this tension, we extend the classical definitions of the Jacobian and Picard group of $\spz$ to incorporate the infinite nature of the arithmetic curve (\S\ref{gendiv}). The new arithmetic Jacobian is the natural ambient space for studying étale covers—both geometrically and adelically—via pullbacks using an analogue of the Abel-Jacobi map.

The first step in this construction is to extend the notion of a divisor on $\spz$. A classical divisor is a formal sum $D = \sum n_p [p]$ where the coefficients $n_p$ are integers and vanish for all but finitely many $p$. To accommodate the expected infinite genus, we relax these conditions in two ways.\vspace{.02in}

1) We allow the coefficients $n_p$ to take values in  $\Z \cup \{\infty\}$\vspace{.02in}

    and\vspace{.02in}
    
2) We allow the support of a divisor to be infinite, imposing only the condition that $n_p$ is negative for at most finitely many primes.\vspace{.02in}

This extended definition of divisors is justified by a canonical isomorphism  with the lattice of rank-1 subgroups of $\Q$ (Proposition~\ref{rankpro}). To any such generalized divisor $D = \sum n_p [p]$, one associates the module of global sections:
\[
\mathcal L(D) = \{ q \in \Q \mid v_p(q) \ge -n_p \quad \forall p \}.
\]
Conversely, any rank-1 subgroup $L \subset \Q$ determines a divisor $D$ by setting \sloppy $n_p = -\inf \{ v_p(x) \mid x \in L \}$, and one shows that $L=\mathcal L(D)$. 
In this framework, the classical notion of linear equivalence of divisors corresponds precisely to the isomorphism of the associated subgroups of $\Q$. Furthermore, the addition of divisors corresponds to the tensor product of these groups:
$
\mathcal L(D_1 + D_2) \cong \mathcal L(D_1) \otimes_\Z \mathcal L(D_2)
$.

The \emph{Picard monoid}  of $\spz$ is defined as the  quotient monoid
\[
\Pic(\spz)=\Div(\Spec(\Z)) / \PDiv(\Spec(\Z)).
\]
The above identification in terms of subgroups of $\Q$ provides  an algebraic description of  \(\Pic(\spz)\) as the (multiplicative) monoid of isomorphism classes of \emph{torsion-free rank-$1$ groups}.
In turn, this algebraic description has two equivalent incarnations: adelic and  geometric.
First, in terms of adeles, there is a monoid  isomorphism (Proposition~\ref{ident})
\[
\Pic(\spz)\cong \Q^\times \backslash \Af / \znt,
\]
where $\Af$ is the ring of finite rational adeles. Here, the monoid structure on $\Pic(\spz)$  corresponds to the product of adeles.
Second, $\Pic(\spz)$ can be identified   with the set of points of the arithmetic topos $\nat$ \cite{CC1}, which is dual to the multiplicative monoid $\N^\times$ of non-zero positive integers. The monoid structure on $\Pic(\spz)$ corresponds to the product of points in the topos induced by the multiplication map $\N^\times \times \N^\times \to \N^\times$ (Appendix \ref{appB}).

To capture the full geometry, we pass to the completion \(\spzb\), obtained by adjoining the archimedean place.
In section~\ref{extended} we extend the definition of divisors  by including, as in Arakelov geometry, a norm on the associated real lattice $L\otimes_\Z\R$. 
We name the resulting monoid of divisor classes, the \emph{arithmetic Picard}  $\Pic(\spzb)$. This is the moduli space of \emph{arithmetic divisors} $\mathcal D=(L,\Vert\cdot\Vert)$, where $L$ is a torsion free rank $1$ group, and $\Vert\cdot\Vert$ is a semi-norm on $L\otimes_\Z \R$. 
The group $\R_+^\times$ acts on $\Pic(\spzb)$  by rescaling of the norm. 

The monoid $\Pic(\spzb)$ has an adelic interpretation in terms of the monoid
\begin{equation}\label{riemsct}
\XQ = \Q^\times \backslash \A / \znt,
\end{equation}
the quotient of the rational adele class space $Y_\Q=\Q^\times \backslash \A$ by the maximal compact subgroup  $\znt\subset \CQ= \Q^\times \backslash \A^\times$ of the idele class group (\S \ref{adrel}):
\begin{thm}\label{xqthm}
The map  which  associates to a rational adele $a=(a_f,a_\infty)$ the pair consisting of:
\begin{itemize}[leftmargin=*, labelindent=0pt]
    \item[-] The group $L_a := \{q \in \Q \mid a_f q \in \widehat{\Z}\}$,
    \item[-] The semi-norm $\|\cdot\|_a$ on $L_a \otimes_\Z \R$ defined by $\|v\|_a = |a_\infty| \cdot |v|$ (via the canonical identification $L_a \otimes \R \cong \R$),
\end{itemize}
induces a canonical monoid isomorphism
\[
X_\Q \cong\Pic(\spzb)
\]
between the Riemann sector $X_\Q$ of the adele class space \eqref{riemsct} 
and the moduli space of isomorphism classes of arithmetic divisors $\mathcal D=(L, \|\cdot\|)$.
\end{thm}
The main feature of the above isomorphism 
is to give to the adelically defined space $X_\Q$ a geometric meaning as a moduli space of arithmetic divisor classes.

In section~\ref{jacob} we introduce  the \emph{arithmetic Jacobian} as the quotient
$$\Jac(\spzb)=\Pic(\spzb)/\R_+^\times.
$$
This definition  parallels the classical notion of the Jacobian of a curve as the quotient of its Picard group by the subgroup generated by a divisor of degree $1$.

The \emph{extended arithmetic Abel-Jacobi map} $\Theta$ is defined in \S \ref{sheaf} by sending a prime $p$ to the divisor $\infty \cdot [p]$, the archimedean place $\infty$ to the divisor $\infty \cdot [\infty]$, and the generic point $\eta$ to the zero divisor. In formulas one writes:
\[
\Theta: \spzb \longrightarrow \Jac(\spzb), \quad\Theta(x)= \begin{cases}(\mathbf{Z}[1 / p],\|\cdot\|\neq 0) & \text { if } x=p<\infty  \\ (\mathbf{Z}, 0) & \text { if } x=\infty \\ (\mathbf{Z},\|\cdot\|\neq 0) & \text { if } x=\eta.\end{cases}
\]
The non-vanishing of the norm ensures the uniqueness of the arithmetic divisor class. This map extends canonically to the infinite symmetric power of $\spzb$, mirroring the classical construction in algebraic geometry. In the function field case, the Abel-Jacobi map sends the $g$-th symmetric power of the curve to the Jacobian, which is an abelian variety of dimension $g$. Here, the arithmetic curve has infinite genus ($g=\infty$), and the target space is not an abelian variety but an idempotent semilattice. The extension yields a surjection from the infinite symmetric power onto the set of idempotents of $\Jac(\spzb)$, realizing the boundary structure of the Picard monoid as the limit of the symmetric powers of the curve.

\vspace{.05in}

In Section \ref{sectsheaf}, we connect the abstract theory of the arithmetic  Picard $\Pic(\spzb)$ to the sheaf-theoretic framework of $\F1$-algebras as in \cite{CC2}. 
\begin{thm}
An arithmetic divisor $\mathcal{D} = (L, \|\cdot\|)$ defines a sheaf of modules $\mathcal{O}(\mathcal{D})$ over the structure sheaf $\mathcal O$ of the \emph{absolute curve} $\spzb$.
\begin{itemize}[leftmargin=*, labelindent=0pt]
    \item[-] At the finite primes, the sheaf $\mathcal{O}(\mathcal{D})$ is constructed using the Eilenberg-MacLane functor $H$ applied to the group of regular sections of $L$.
    \item[-] At the archimedean place, the metric $\|\cdot\|$ imposes the "unit ball" condition $$\sum_{x \neq\{*\}} \|\phi(x)\| \le 1$$ on the sections.
\end{itemize}
\end{thm}
A key feature of this result is that it unifies the treatment of classical Arakelov divisors (\ie with finite support) and the new arithmetic divisors (allowing infinite type). For the latter, the vanishing of the semi-norm implies that the metric constraint becomes vacuous, geometrically corresponding to allowing a pole of arbitrary order at infinity.\vspace{.05in}

The  arithmetic Jacobian  $\Jac(\spzb)$ is the base of a tower of adelic coverings.
At the top of this tower sits the adele class space $\YQ = \Q^\times \backslash \A$, carrying a natural action of the idele class group $\CQ$, while $\XQ$ fits-in as an important intermediate cover.

In section~\ref{sectpullback} we develop the construction of the corresponding coverings  of $\spzb$ via the pullback\footnote{This construction is  analogous to the classical geometric class field theory for a curve $C$ over a finite field. There, one constructs abelian coverings of $C$ by pulling back coverings of the Jacobian $J(C)$. Specifically, one considers the quotient $J(C) \to J(C)/\ell^n J(C)$, where the Galois group (the $\ell^n$-torsion points) acts by translation on the Jacobian. } of the adelic covers along the extended Abel-Jacobi map 
\[
\begin{tikzcd}
\mathcal{X}_\Q \arrow[r] \arrow[d] & X_\Q \arrow[d,two heads] \\
\spzb \arrow[r, "\Theta"] & \Jac(\spzb)
\end{tikzcd}
\]
The space $\mathcal{X}_\Q$ obtained by  pullback of the adelic $X_\Q$ along $\Theta$ coincides with what we termed in \cite{CC3} the "visible part" of the Riemann sector $X_\Q$. It consists of the union of the periodic orbits of the flow over the finite primes, the fixed point over the archimedean place and the generic orbit $\R^\times_+$ over the generic point $\eta$. The fiber of the  covering  $\pi$ over a finite prime $p$ is  the periodic orbit $C_p$, and the monoid structure of $X_\Q\cong \Pic(\spzb)$  coincides with the group structure of $C_p \cong \R / (\log p)\Z$.\vspace{.03in}

In Section~\ref{sectuniv} we  move to the ``universal level''. First, we describe the Galois correspondence  between intermediate covers of $\Jac(\spzb)$ and   closed subgroups of $\CQ$. The action of the "Galois group" $\CQ$ on these spaces is by "translation", which in this monoid context means multiplication by the invertible elements (ideles) acting on the space of adele classes.

In  \S \ref{sectcft} we first recall the construction (Theorem~\ref{main} in \cite{CC3}) of a functor from finite abelian extensions of $\Q$ to finite covers of the space $X_\Q$.
Then, in \S \ref{fibers} we consider the pullback, along the Abel-Jacobi map $\Theta$, of the covering $\YQ \to \XQ$. This adds a top layer to the previous tower, yielding a space $\widetilde{\spz}$ that covers $\mathcal{X}_\Q$
\[
\begin{tikzcd}
\widetilde{\spz} \arrow[r,"\tilde\Theta"]\arrow[d] & Y_\Q\arrow[d]
\\
\mathcal{X}_\Q \arrow[r]  & X_\Q 
\end{tikzcd}
\]

  The geometry of the universal cover $\widetilde{\spz}$ over $C_p\subset \mathcal X_\Q$  (\ie the fiber over $C_p$) captures the interaction between the prime $p\in \spzb$ and all other primes $\ell \neq p$.
Specifically, the restriction of $\widetilde{\spz}$ to the orbit $C_p$ is the mapping torus of the Frobenius automorphism acting on the local units away from $p$. Let $
K^{(p)} = \prod_{\ell \neq p} \Z_\ell^\times$
be the group of units of the adeles omitting the component at $p$. Then the fiber over $C_p$ is identified with the quotient:
\[
\mathcal{T}_p = \left(K^{(p)} \times \R\right)/\Big((u, t) \sim (p u, t + \log p)\Big).
\]
Here, the generator of the fundamental group of the orbit $C_p$ acts on the fiber $K^{(p)}$ by multiplication by $p$. This action corresponds, via the Artin map, to the action of the Frobenius at $p$ on the abelianized \'etale fundamental group of the curve $\spzb$ localized at $p$. Thus, the universal geometry over $p$ encodes the linking of $p$ with all other primes $\ell$.

The following commutative diagram of pullbacks helps to visualize the overall construction
\begin{equation}\label{3floora}
\begin{tikzcd}
\widetilde{\spz} \arrow[r,"\tilde\Theta"]\arrow[d] & Y_\Q\arrow[d]
\\
\mathcal{X}_\Q \arrow[r] \arrow[d] & X_\Q \arrow[d] \\
\spzb \arrow[r, "\Theta"] & \Jac(\spzb)
\end{tikzcd}
\end{equation}

The spaces $\widetilde{\spz}$ and $\mathcal{X}_\Q$ in the left column, are  defined as pullbacks of adelic quotients. 

\begin{thm}\label{thm:fibers}
The monoid structure of the adelic covers  $X_\Q$ and $Y_\Q$ induces a canonical group structure on the fibers of the pullbacks of these covers along the arithmetic Abel-Jacobi maps $\Theta$ and $\tilde\Theta$. The structure of these fibers for the intermediate cover $\mathcal{X}_\Q$  and the universal cover $\widetilde{\spz}$ is given as follows:
\begin{center}
\renewcommand{\arraystretch}{1.5}
\begin{tabular}{c|c|c}
$x \in \spzb$ & \rm{Fiber in} $\mathcal{X}_\Q$ & \rm{Fiber in} $\widetilde{\spz}$ \\
\hline
$\eta$ & $\R_+^\times$ & $C_\Q$  \\
$p<\infty$ & $C_p \cong \R / \Z \log p$ & $C_\Q / \Q_p^\times$ \\
$\infty$ & $\{1\}$ & $\widehat{\Z}^\times \cong \operatorname{Gal}(\Q^{\text{ab}}/\Q)$ \\
\end{tabular}
\end{center}
\end{thm}
In Section \ref{sectgeompic} we establish a rigorous geometric realization of the adelic monoid $Y_\Q = \Q^\times \backslash \A$, transforming it from a purely adelic quotient into a moduli space of arithmetic objects equipped with \emph{rigidifying data}. Our construction proceeds in two stages, revealing the ``combinatorial skeleton'' hidden within the adelic structure.

First, we introduce the notion of \emph{framed arithmetic divisors}. These are torsion free rank-1 groups $L$ equipped with trivializations of their completions at both finite and infinite places: that is, a generator $\xi$ of the $\Zhat$-module $\operatorname{Hom}(L, \Zhat)$ and a homomorphism $\tau: L \to \R$. This formulation places the archimedean and non-archimedean data on an equal footing and allows one to characterize the rigidity of the frame.

Second, by applying Pontryagin duality to the finite frame, we pass to the dual description of \emph{rooted arithmetic divisors}. Here, the profinite data is replaced by a \emph{root} \ie a homomorphism $\rho$ from the group of roots of unity $\Q/\Z$ to the torsion of the dual group $L^\vee = \operatorname{Hom}(L, \Q/\Z)$. This root map encodes the arithmetic position of the divisor by specifying how the standard roots of unity map into the ``arithmetic roots'' of $L$. We show that the multiplicative structure of the adeles corresponds geometrically to the tensor product of these divisors. A key technical innovation is  the definition of the tensor product of roots, which relies on a ``hidden'' coproduct structure on $\Q/\Z$ arising from the limit of the ring structures of the levels $\Z/n\Z$.\vspace{.05in}

We denote by $\MFr(\spzb)$  the \emph{moduli space of framed arithmetic divisors} $(L, \xi, \tau)$ modulo isomorphism, and by $\MRt(\spzb)$  the \emph{moduli space of rooted arithmetic divisors} $(L, \rho, \tau)$ modulo isomorphism.\vspace{.03in}

The central result arising from this analysis is the following uniformization theorem (Theorem \ref{unif})

\begin{thm}[Geometric uniformization of the adelic monoid $Y_\Q$]
\,
\begin{enumerate}
    \item There is a canonical monoid isomorphism between the adelic monoid $Y_\Q$ and the moduli space of framed arithmetic divisors:
    \[
    \mathcal{F}: Y_\Q \xrightarrow{\sim} \MFr(\spzb), \quad [a] \longmapsto [(L_a, \xi_a, \tau_a)].
    \]
    Under this isomorphism, the product of adeles corresponds to the tensor product of framed divisors:
    \[
    \mathcal{F}(ab) \cong \mathcal{F}(a) \otimes \mathcal{F}(b).
    \]
    
    \item The duality functor $\xi\mapsto \xi^\vee$ induces a canonical isomorphism of monoids between framed and rooted arithmetic divisors:
    \[
    D: \MFr(\spzb) \xrightarrow{\sim} \MRt(\spzb), \quad (L, \xi, \tau) \longmapsto (L, \xi^\vee, \tau).
    \]
    This map respects the tensor structure, where the product of roots is defined via the diagonal action on the system of finite rings $\Z/n\Z$.
\end{enumerate}
\end{thm}

Consequently, the adelic monoid $Y_\Q$ classifies rooted arithmetic divisors, identifying the multiplication of adelic classes with the tensor product of systems of roots.

The duality isomorphism $D$ justifies the introduction of the following notation
\begin{equation}\label{}
\tspz := \MFr(\spzb) \cong \MRt(\spzb)
\end{equation}

In Section \ref{Wtrace}, we draw a parallel with classical algebraic geometry, where the Jacobian and Picard varieties provide the intrinsic representation of cohomology. We re-interpret the diagram \eqref{3floora} as the key pullback diagram  in the construction of \'etale covers in algebraic geometry
\begin{equation}\label{}
\begin{tikzcd}
\widetilde{\spz} \arrow[r,"\tilde\Theta"]\arrow[d,"\tilde\pi"'] & \tspz\arrow[d,"\pi"]
 \\
\spzb \arrow[r, "\Theta"] & \Jac(\spzb)
\end{tikzcd}
\end{equation}

We investigate  in the arithmetic context over $\spzb$, the role of the  monoid $\tspz$  as a geometric support for the cohomology of the arithmetic curve. The nature of this cohomology is revealed by the spectral realization of the zeros of $L$-functions \cite{Co-zeta}. In \S \ref{sect8.1} we interpret  Weil's explicit formula as a Lefschetz trace formula for the translation action of the idele class group $C_\Q$ on $\tspz$, demonstrating that the geometric contributions of the periodic orbits localize precisely on the range of the universal Abel-Jacobi map $\tilde\Theta$. This marks a decisive departure from the classical function field analogy. There, the spectral theory arises from the action of the Frobenius endomorphism on the Jacobian. Here,  we rely on the action of the idèle class group  acting by \emph{translations} as the Galois group of a covering. While translations on a  group possess no non-trivial fixed points, the adelic monoid $\tspz$ carries a geometric structure where the action of $C_\Q$ by translation admits fixed points, recovering the local terms of the explicit formula.

In \S \ref{geomtrace} we show that the geometric side of this trace formula decomposes according to the fibers of the projection $\tilde\pi: \widetilde \spz \to \spzb$ described above. Precisely:\vspace{.03in}

    The fibers over the closed points (places) $v$ ($p<\infty$ and $\infty$) of $\spzb$ carry the local terms of the explicit formula. The local multiplicative groups $\Q_v^\times$ appear as isotropy subgroups of the action, while the term $\frac{1}{|1-u|}$ in the generalized Weil's explicit formula  arises as the distributional trace of the scaling operator on the transverse space to the orbit.
    
    The fiber over the generic point $\eta\in \spzb$ contributes to the spectral side via the regular representation, manifesting as the "white light" divergence in the trace.

To provide a rigorous formulation, we follow  \cite{Co-zeta} and develop in \S \ref{sectsemi} a semi-local trace formula.  In \S \ref{sectshespec} we introduce a sheaf of Schwartz algebras on $\spz$ associated to finite products of local fields. The  semi-local trace formula of \cite{Co-zeta} involves phase-space cutoffs $R_\lambda$ acting on the associated Hilbert spaces, and  contains a divergent term $2h(1)\log \lambda$  (see \eqref{semilocTrace})  while the finite remainder terms recover exactly the local integrals in Weil's explicit formula. We interpret the divergent term as the contribution of the image $\Theta(\eta)=\tilde \eta$ of the generic point $\eta\in \spzb$ by the Abel-Jacobi map. This suggests that the relevant cohomology is that of the pair $(\tspz,\tilde\eta)$.\vspace{.05in}

Finally, we emphasize that the passage from the classical geometry of Picard \emph{groups} to that of \emph{Picard monoids} (and likewise for Jacobians) should not be viewed as a mere generalization. Rather, it is a necessary step for incorporating boundary and singular phenomena into the geometric framework—precisely the features required for the spectral realization, including the ``white light’’ and the distributional trace. This perspective is also consistent with developments in modern geometric representation theory and the Langlands program, where \emph{Vinberg monoids} provide a natural and well-established precedent (see \cite{Vinberg, Ngo}).
\vspace{.08in}

In Appendix \ref{appendual}, we give a self-contained analysis of the character dual $L^\vee = \operatorname{Hom}(L, \Q/\Z)$ for torsion-free rank-1 groups. We then derive an explicit adelic presentation of this group as the quotient $\A_f / L^\perp$, where $L^\perp$ denotes the annihilator of $L$ in the ring of finite rational adeles. A key result identifies the torsion subgroup $\rm{Tor}(L^\vee)$ as the direct sum of the local components $\Q_p/\Z_p$ over the primes where $L$ is integral, regardless of the group's behavior at singular primes. Finally, we clarify the relationship between $L^\vee$ and the classical Pontryagin dual $L^* = \operatorname{Hom}(L, S^1)$, demonstrating that $L^\vee$ contains the essential torsion information required to define the system of roots.

Appendix \ref{appB} develops the identification of $\Pic(\spz)$   with the monoid of the points of the arithmetic topos $\nat$.


\section{Arithmetic divisors on $\spz$}\label{gendiv}

This section develops the divisor theory of $\spz$ starting from supernatural numbers
and valuation bounds. We define arithmetic divisors and their spaces of global
sections $\mathcal L(D)\subset\Q$, obtaining a classification of rank-$1$ subgroups
of $\Q$ and, modulo linear equivalence, an identification of $\Pic(\spz)$ with
isomorphism classes of torsion-free rank-$1$ groups. We then describe the induced
tensor/monoid structure, give an adelic realization of $\Pic(\spz)$, and introduce an
arithmetic Abel--Jacobi map landing on a boundary (idempotent) stratum.

The structural results regarding rank-1 torsion-free abelian groups presented in this section are standard in the literature; specifically, the classification via ``types'' (which corresponds to our divisor classes) is treated in Rotman \cite[Theorems 10.47 and 10.48]{Rotman}. However, to establish our specific geometric terminology and for the convenience of the reader, we provide self-contained proofs adapted to the language of arithmetic divisors.

\subsection{Supernatural numbers}

To motivate the definition of arithmetic divisors on $\spz$, we first recall the
notion of supernatural numbers.
These encode valuation bounds at all primes simultaneously and provide a natural
language for describing divisor-like data with infinite support but finite negative
part.

\begin{defn}
The set of supernatural numbers $\N^{sup}$ consists of formal products:
\[
N = \prod_{p} p^{n_p}
\]
where the exponents satisfy $n_p \in \N \cup \{+\infty\}$.
  \end{defn}
To every supernatural number $N \in \N^{sup}$, we associate a subgroup $H_N \subset \Q$:
\[
H_N = \{ x \in \Q \mid v_p(x) \ge -n_p \text{ for all primes } p \}
\]
Here, $v_p(x)$ is the standard $p$-adic valuation. 
The condition $v_p(x) \ge -n_p$ can be interpreted geometrically as follows:
\begin{itemize}
    \item[] If $n_p > 0$, elements in $H_N$ may have poles of order at most $n_p$ at $p$ (\ie denominators divisible at most by $p^{n_p}$).
       \item[] If $n_p = +\infty$, elements in $H_N$ may have poles of arbitrary order at $p$.
\end{itemize}

The next statement explains why supernatural numbers are the natural parameters in this setting.
They are equivalent to prescribing, prime by prime, the denominators that are allowed in a subgroup
of $\Q$ containing $\Z$  
(cf. \cite{Serre}, Chap. XIII, \S 2).

\begin{prop}\label{rankpro}
The map $N \mapsto H_N$ defines a bijection between $\N^{sup}$ and the set of all  subgroups of $\Q$ containing $\Z$.
\end{prop}

\begin{proof}
For $N=\prod_p p^{n_p}$ the set
\[
H_N=\{x\in\Q\mid v_p(x)\ge -n_p\ \forall p\}
\]
is clearly an additive subgroup of $\Q$ and contains $\Z$, so the map is well-defined.

Conversely, let $H\subset\Q$ be a subgroup with $\Z\subset H$. For each prime $p$ set
\[
n_p(H)=\sup\{-v_p(x)\mid x\in H\}\in \N\cup\{+\infty\},
\quad
N(H)=\prod_p p^{n_p(H)}\in \N^{\sup}.
\]
Then $H\subseteq H_{N(H)}$ is immediate from the definition of $n_p(H)$.

For the reverse inclusion, take $x\in H_{N(H)}$ and write $x=a/b$ in lowest terms, with
$b=\prod_p p^{k_p}$ and $k_p=-v_p(x)$. The condition $x\in H_{N(H)}$ means $k_p\le n_p(H)$ for
all $p\mid b$. For each such $p$ choose $x_p\in H$ with $-v_p(x_p)\ge k_p$; multiplying by a suitable
integer (allowed since $\Z\subset H$) we may assume $x_p=u_p/p^{k_p}$ with $u_p\in\Z$ and $p\nmid u_p$.
Let $S$ be the finite set of primes dividing $b$, and put $b_S:=\prod_{p\in S}p^{k_p}=b$.

Choose integers $c_p$ such that
\[
\sum_{p\in S} c_p\,u_p\,\frac{b}{p^{k_p}}\equiv a \pmod{b}.
\]
This is possible because each factor $\frac{b}{p^{k_p}}$ is invertible modulo $p^{k_p}$, hence one
can solve the congruences modulo each $p^{k_p}$ and combine them by the Chinese remainder theorem.
Then, define
\[
y=\sum_{p\in S} c_p\,x_p \in H.
\]
By construction, $y\equiv a/b \pmod{\Z}$, \ie $y-\frac{a}{b}\in\Z\subset H$, hence $x=a/b\in H$.
Therefore $H_{N(H)}\subseteq H$, so $H=H_{N(H)}$.

Finally, the assignments $N\mapsto H_N$ and $H\mapsto N(H)$ are inverse to each other by construction,
so the correspondence is bijective.
\end{proof}

\subsection{Arithmetic divisors on $\spz$}

We begin by extending the classical notion of a divisor to accommodate the "infinite genus" nature of $\spz$.

\begin{defn}\label{gendiv}
An arithmetic divisor on $\spz$ is a formal sum:
\[
D = \sum_{p} n_p [p]
\]
where the coefficients $n_p$ belong to $\Z \cup \{+\infty\}$ and satisfy the condition that $n_p < 0$ for at most finitely many rational primes $p$.
\end{defn}

We denote by $\operatorname{Div}(\spz)$ the set of arithmetic divisors. An arithmetic divisor is called \emph{effective} if $n_p \ge 0$ for all $p$. \vspace{.03in}

The set of effective arithmetic divisors is naturally identified with the set of
supernatural numbers $\N^{\sup}$, via the correspondence
\[
\prod_p p^{n_p}\ \longleftrightarrow\ \sum_p n_p [p],
\]
which amounts to passing from the multiplicative to the additive encoding of
valuation data. This identification reflects the fact that both notions encode, for each prime $p$,
the maximal allowed order of a pole, with no finiteness restriction on the positive
part.\vspace{.03in}

To every arithmetic divisor $D$, we associate the module of global sections of the corresponding line bundle. This is the subgroup $\mathcal L(D) \subset \Q$ defined by:
\[
\mathcal L(D) = \{ x \in \Q \mid v_p(x) \ge -n_p \text{ for all primes } p \}.
\]
Here, $v_p(x)$ is the standard $p$-adic valuation. The condition $v_p(x) \ge -n_p$ has a clear geometric interpretation:
\begin{itemize}[leftmargin=*, labelindent=0pt]
    \item[] If $n_p > 0$ (finite), sections in $\mathcal L(D)$ are allowed to have poles of order at most $n_p$ at $p$.
    \item[] If $n_p = +\infty$, the sections are allowed to have poles of arbitrary order at $p$ (essential singularity).
    \item[] If $n_p < 0$, the sections must have a zero of order at least $|n_p|$ at $p$.
\end{itemize}

As in the classical theory of divisors on curves, a divisor determines the space of
global sections of the associated line bundle.
In the present arithmetic setting, this assignment completely characterizes
rank-1 subgroups of $\Q$, as stated in the following proposition.

\begin{prop}\label{rankpro}
The map $D \mapsto \mathcal L(D)$ defines a bijection between the set of generalized divisors $\operatorname{Div}(\spz)$ and the set of rank-1 subgroups of $\Q$.
\end{prop}

\begin{proof}
    The map is well-defined: for an arithmetic divisor
$D=\sum_p n_p[p]$, the set $\mathcal L(D)$
is clearly an additive subgroup of $\Q$. Since $n_p<0$ for only finitely many primes,
there exists a nonzero $x\in\Q$ satisfying these inequalities, so $\mathcal L(D)$ has
rank~$1$.

Conversely, let $H\subset\Q$ be a rank-$1$ subgroup. For each prime $p$, define
\[
n_p:=-\inf\{v_p(x)\mid x\in H\}\in\Z\cup\{+\infty\},
\]
and set $D_H:=\sum_p n_p[p]$.
Since $H\subset\Q$, for any fixed $p$ the set $\{v_p(x)\mid x\in H\}$ is bounded below,
so $n_p>-\infty$.
Moreover, choose a nonzero element $q\in H$.
Then for all but finitely many primes $p$ we have $v_p(q)\ge 0$, which implies
$v_p(x)\ge 0$ for all $x\in H$ and hence $n_p\ge 0$.
Thus $D_H$ is an arithmetic divisor.

By construction, an element $x\in\Q$ lies in $\mathcal L(D_H)$ if and only if
$v_p(x)\ge \inf\{v_p(y)\mid y\in H\}$ for all $p$, which is equivalent to $x\in H$.
Hence $\mathcal L(D_H)=H$, and the correspondence is bijective.
\end{proof}

\begin{rem}
Under the bijection stated in the proposition:
\begin{itemize}[leftmargin=*, labelindent=0pt]
    \item[-] The effective divisors (supernatural numbers) correspond to subgroups $H\subset\Q$ such that $\Z \subseteq H$.
    \item[-] The classical divisors (finite support, integer coefficients) correspond to fractional ideals of $\Z$.
    \item[-] The divisor $D_p = \infty \cdot [p]$ supported solely at $p$ corresponds to the group $\Z[1/p]$.
\end{itemize}
\end{rem}

\subsection{Linear equivalence and the Picard monoid}

A non trivial abelian group is isomorphic to a subgroup of the additive group $\Q$
if and only if it is torsion free and of rank-1.
Moreover, two nonzero subgroups of $\Q$ are isomorphic if and only if they differ
by multiplication by a nonzero rational number.
Explicitly, for $M,N\in \N^{sup}$:
\[
H_M \cong H_N
\iff
\exists\, q \in \Q^\times \text{ such that } H_M = q\,H_N .
\]

In valuation-theoretic terms, an element $x\in\Q$ belongs to $qH_N$ if and only if
$x/q\in H_N$, that is,
\[
v_p(x)-v_p(q)\ge -n_p \quad \text{for all primes } p.
\]
Equivalently, the corresponding valuation bounds satisfy
\[
m_p = n_p - v_p(q),
\]
with the convention $\infty + k = \infty$.
In divisor language, this is precisely the relation of linear equivalence.
The two arithmetic divisors $D_M=\sum_p m_p[p]$ and $D_N=\sum_p n_p[p]$ satisfy
\[
D_M \sim D_N
\iff
D_M = D_N - \rm{div}(q),
\]
where $\rm{div}(q)=\sum_p v_p(q)[p]$ denotes the principal arithmetic divisor associated
to $q\in\Q^\times$.

\begin{defn} The arithmetic Picard monoid of $\spz$ is  the quotient
\[
\Pic(\spz)
=
\Div(\spz)\big/ \PDiv(\spz),
\]
where $\PDiv(\spz)$ denotes the subgroup of principal arithmetic divisors.
\end{defn}

\begin{prop}\label{rankone}
The assignment $D\mapsto \mathcal L(D)$ induces a bijection between the arithmetic
Picard monoid $\Pic(\spz)$ and the set of isomorphism classes of torsion-free
rank-1 abelian groups.
\end{prop}

\begin{proof}
This is an immediate consequence of the correspondence between arithmetic divisors
and rank-1 subgroups of $\Q$, together with the identification of linear equivalence
with scaling by $\Q^\times$.
\end{proof}

\begin{rem}\label{choice} 
Assuming the Axiom of Choice, the set of isomorphism classes of torsion-free
rank-1 abelian groups has cardinality equal to that of the continuum.
Nevertheless, Proposition~\ref{rankone} shows that this space admits no effective
classification by real-valued invariants.
More precisely, although its set-theoretic cardinality is that of $\R$, the
\emph{effective cardinality} of this quotient is strictly larger than the continuum.
This phenomenon is a consequence of the ergodic nature of the linear equivalence
relation on $\Div(\spz)$.
Indeed, one can equip $\Div(\spz)$ with a probability measure for which the linear
equivalence relation is ergodic\footnote{That is, any measurable subset of
$\Div(\spz)$ that is a union of equivalence classes has measure $0$ or $1$.}.
It follows that any measurable map from the quotient space to $\R$ is almost
everywhere constant, and therefore cannot separate equivalence classes.

It follows that the moduli space of rank-1 torsion-free groups lies beyond the
scope of classification by classical invariants.

\end{rem}

\subsection{Tensor product and addition of divisor classes}

As in algebraic geometry addition of divisors induces a monoidal structure on divisor
classes, in the present arithmetic setting this operation admits a concrete
interpretation in terms of tensor products of rank-1 subgroups of $\Q$.

One defines the sum of two arithmetic divisors $D=(n_p)$ and $D'=(n'_p)$, by the rule 
$$
(D+D')_p=n_p+n'_p.
$$
This operation induces a monoid structure on $\Pic(\spz)$ that is interpreted in terms of the associated groups.

\begin{prop}
Let $D_1$ and $D_2$ be arithmetic divisors on $\spz$, and let
$H_1,H_2\subset\Q$ be the associated rank-1 subgroups.
Then the subgroup associated to the sum $D_1+D_2$ is
\[
H_{D_1+D_2}=H_1\cdot H_2 \;\cong\; H_1\otimes_{\Z} H_2,
\]
where the tensor product is canonically identified with the product inside $\Q$.
\end{prop}

\begin{proof}
Recall that for an arithmetic divisor $D=\sum_p n_p(D)[p]$ one has
\[
H_D=\{x\in\Q \mid v_p(x)\ge -n_p(D)\ \text{for all }p\}.
\]
For rank-1 subgroups of $\Q$, the tensor product over $\Z$ is canonically identified
with the subgroup generated by products:
\[
H_1\otimes_{\Z} H_2 \;\cong\; H_1\cdot H_2
:= \Bigl\langle xy \mid x\in H_1,\; y\in H_2 \Bigr\rangle \subset \Q.
\]

We first show the inclusion $H_1\cdot H_2 \subseteq H_{D_1+D_2}$.
Let $x\in H_1$ and $y\in H_2$. For every prime $p$,
\[
v_p(x)\ge -n_p(D_1), \qquad v_p(y)\ge -n_p(D_2),
\]
hence
\[
v_p(xy)=v_p(x)+v_p(y)\ge -\bigl(n_p(D_1)+n_p(D_2)\bigr)
= -n_p(D_1+D_2).
\]
Since the $p$-adic valuation of a sum is bounded below by the minimum of the valuations
of its summands, the same inequality holds for any finite $\Z$-linear combination of
such products. Thus $H_1\cdot H_2\subseteq H_{D_1+D_2}$.

For the reverse inclusion, let $z\in H_{D_1+D_2}$.
Write $z=\pm\prod_p p^{-k_p}$ with $k_p=-v_p(z)$.
The condition $z\in H_{D_1+D_2}$ means
\[
k_p \le n_p(D_1)+n_p(D_2)\qquad \text{for all }p.
\]
For each prime $p$, choose integers $k_{1,p}\le n_p(D_1)$ and
$k_{2,p}\le n_p(D_2)$ such that $k_p=k_{1,p}+k_{2,p}$
(which is always possible, with the convention $\infty+k=\infty$).
Define
\[
x:=\pm\prod_p p^{-k_{1,p}}\in H_1,
\qquad
y:=\prod_p p^{-k_{2,p}}\in H_2.
\]
Then $z=xy$, so $z\in H_1\cdot H_2$.
This proves $H_{D_1+D_2}\subseteq H_1\cdot H_2$.

Therefore $H_{D_1+D_2}=H_1\cdot H_2\cong H_1\otimes_{\Z} H_2$, and the addition of
arithmetic divisors corresponds to the tensor product of the associated rank-one
groups.
\end{proof}

\subsection{Adelic description of arithmetic divisors}

Let $\A_f$ denote the ring of finite rational adeles, and let
$\widehat{\Z}\subset \A_f$ be the compact subring of profinite integers,
so that $\widehat{\Z}^\times=\prod_p \Z_p^\times$.
To a finite adele $a=(a_p)\in\A_f$ we associate the family of valuations
\[
\tilde\Phi(a):=(v_p(a_p))_{p},
\]
with the convention $v_p(0)=\infty$.
Since $a_p\in\Z_p$ for almost all primes $p$ (by definition of the restricted product in adeles), one has $v_p(a_p)\ge 0$
for almost all $p$, so $\tilde\Phi(a)$ defines an arithmetic divisor.

Furthermore, multiplication by a unit $u=(u_p)\in\widehat{\Z}^\times$ does not affect
valuations, since $v_p(u_p)=0$ for all $p$.
Hence $\tilde\Phi$ factors through the quotient and induces a map
\begin{equation}\label{fifi}
\Phi:\A_f/\hat{\Z}^\times \longrightarrow \Div(\Spec(\Z)).
\end{equation}
The map $\Phi$ is surjective: given an arithmetic divisor
$D=(n_p)$, the adele $a=(a_p)$ with $a_p=p^{n_p}$ (with the convention
$p^\infty=0$) is a preimage.
It is also injective because two elements of $\Q_p$ with the same valuation
differ by multiplication by a unit in $\Z_p^\times$. It then follows that $\Phi$ is a bijection.
Quotienting further by the diagonal action of $\Q^\times$, one obtains
a natural bijection
\begin{equation}\label{fifi1}
\Phi:\Q^\times\backslash\A_f/\hat{\Z}^\times
\;\xrightarrow{\ \sim\ }\;
\Div(\Spec(\Z))/\PDiv(\spz)
=\Pic(\spz).
\end{equation}

\begin{rem}\label{ade}
Let $x=(x_p)\in\A_f$.
The rank-1 subgroup of $\Q$ associated to the divisor $\Phi(x)$ is
\[
H=\{q\in\Q \mid xq\in\widehat{\Z}\}.
\]
Indeed, $xq\in\widehat{\Z}$ if and only if $v_p(x_p)+v_p(q)\ge 0$ for all $p$.
\end{rem}

The ring $\A_f$ is a commutative monoid under componentwise multiplication.
This operation is compatible with the equivalence relation defined by the
action of the group
\[
G=\Q_+^\times\times\widehat{\Z}^\times,
\]
and therefore it descends to a well-defined multiplication on the quotient:
\[
[a]\cdot[b]:=[ab].
\]
We derive the following

\begin{prop}\label{ident}
The map $\Phi$ in \eqref{fifi1} is an isomorphism of commutative
monoids: multiplication of finite adeles corresponds to addition of generalized
divisors.
\end{prop}

\begin{proof}
For each prime $p$ the valuation is additive,
\[
v_p(a_p b_p)=v_p(a_p)+v_p(b_p),
\]
with the convention $\infty+x=\infty$.
Thus, if $D_a$ and $D_b$ are the divisors associated to finite adeles $a$ and $b$, one has
\[
D_{ab}=\sum_p\bigl(v_p(a_p)+v_p(b_p)\bigr)[p]=D_a+D_b.
\]
\end{proof}

\subsection{The arithmetic Abel--Jacobi map and boundary points}

In classical algebraic geometry, the Abel-Jacobi map $\phi: C \to \operatorname{Jac}(C)$ sends a point $P$ on a curve $C$ to the isomorphism class of the line bundle $\mathcal{O}(P)$ corresponding to the divisor $[P]$.
On $\Spec(\Z)$ however, the analogous assignment
$p\mapsto[\mathcal O(p)]$ is trivial, since the ideal $p\Z$ is isomorphic to $\Z$
as a rank-1 $\Z$-module.
Thus, in order to obtain a nontrivial arithmetic analogue of the Abel--Jacobi map,
one must allow the image to lie at the boundary of the moduli space of torsion-free
rank-1 groups.

\begin{defn}[Arithmetic Abel--Jacobi map]\label{aAJ}
The \emph{arithmetic Abel--Jacobi map} is the injection
\[
\Theta:\Spec(\Z)\longrightarrow \Pic(\spz)
\]
defined by
\[
\Theta(p)=H_p:=\Z[1/p]\qquad\text{for a prime }p,
\]
and
\[
\Theta(\eta)=\Z,
\]
where $\eta$ denotes the generic point of $\spz$.
\end{defn}

In terms of generalized divisors, the map $\Theta$ sends the generic point $\eta$
to the trivial divisor, and a prime $p$ to the divisor with infinite coefficient
at $p$,
\[
p\longmapsto D_\infty(p):=\infty\cdot[p].
\]

\emph{Injectivity.}
For distinct primes $p\neq q$, the groups $\Z[1/p]$ and $\Z[1/q]$ are not isomorphic:
the former is $p$-divisible while the latter is not.
Hence $\Theta$ separates points of $\Spec(\Z)$.\vspace{.02in}

\emph{Boundary interpretation.}
Although the primes $p$ are interior points of the scheme $\Spec(\Z)$,
their images under $\Theta$ correspond to arithmetic divisors with infinite
multiplicity (representing singularities or poles of infinite order), and therefore lie on the boundary of the Picard space.\vspace{.05in}

Let $C$ be a smooth projective curve of genus $g$ over a field $k$.
Recall that the Jacobian variety $J(C)$ classifies isomorphism classes of degree-zero line bundles on $C$.
The relation between $C$ and $J(C)$ is governed by the Jacobi inversion theorem,
which identifies the $g$-th symmetric power $C^{(g)}$ as a fundamental configuration space
for the Abel--Jacobi map:
\[
\Sigma: C^{(g)} \longrightarrow J(C)
\]
\[
\{P_1, \dots, P_g\} \longmapsto \mathcal{O}(P_1 + \dots + P_g - g P_0)
\]
where $P_0$ is a fixed base point.

In our arithmetic setting, $\spz$ should be regarded as a curve of infinite genus.
Accordingly, one is led to replace the finite symmetric power $C^{(g)}$
by an infinite configuration space.
We model this arithmetic configuration space as follows.

\begin{defn}
Let $\mathcal{P}(\Spec \Z)$ be the set of all subsets $S$ of prime numbers. This corresponds to the set of reduced effective divisors with infinite coefficients:
\[
S \longleftrightarrow D_S = \sum_{p \in S} \infty \cdot [p]
\]
\end{defn}

\begin{defn}[Arithmetic Abel-Jacobi map on configurations]
We define the map $$\Theta: \mathcal{P}(\Spec \Z) \to \Pic(\spz)\qquad 
\Theta(S) = \Z_S = \Z[1/p \mid p \in S]
$$
\end{defn}

\begin{thm}[Tensor structure on configurations]
The map $\Theta$ transforms the union of sets (addition of divisors) into the tensor product of points in the Jacobian (multiplication of subgroups):
\[
\Theta(S_1 \cup S_2) \cong \Theta(S_1) \otimes_\Z \Theta(S_2)
\]
\end{thm}

\begin{proof}
This follows from the universal property of localization:
\[
\Z[S_1^{-1}] \otimes_\Z \Z[S_2^{-1}]
\cong
\Z[(S_1 \cup S_2)^{-1}],
\]
which is the arithmetic analogue of the classical identity
$\mathcal O(D_1+D_2)\cong \mathcal O(D_1)\otimes\mathcal O(D_2)$.
\end{proof}

The map $\Theta$ can thus be viewed as the restriction of the Abel--Jacobi map
to what we call the \emph{idempotent stratum} of the arithmetic Jacobian. In other words:\vspace{.03in}

In $\Pic(\spz)$, the elements $H_p = \Z[1/p]$ are idempotent ($H_p \otimes H_p \cong H_p$).\vspace{.02in}

The map $\Theta$ sends a configuration of primes $\{p_1, \dots, p_k\}$ to the point $H_{p_1} \otimes \dots \otimes H_{p_k}$.
\vspace{.05in}

   For a smooth projective curve $C$ of genus $g$, the Abel--Jacobi map identifies the
$rational$ equivalence class of the symmetric power $C^{(g)}$ with the Jacobian
$J(C)$; in particular, it is surjective, as follows from the Riemann--Roch theorem.
By contrast, in the arithmetic case the curve $\spz$ should be regarded as having
infinite genus, and the corresponding Abel--Jacobi map has image equal to the
idempotent elements of the Picard monoid $\Pic(\spz)$.

\begin{thm}
The image of the map $\Theta: \mathcal{P}(\Spec \Z) \to  \Pic(\Spec\,\Z)$ is exactly the set of idempotent elements of the monoid $\Pic(\Spec\,\Z)$.
\[
\text{Image}(\Theta) = \operatorname{Idem}(\Pic(\Spec\,\Z)) 
\]
\end{thm}

\begin{proof}
We first show that  the image of $\Theta$ consists of idempotents. 

Let $S \subseteq \Spec(\Z)$ be a set of primes. Let $H = \Theta(S) = \Z_S = \Z[S^{-1}]$.
Since $\Z_S$ is a subring of $\Q$, it is closed under multiplication. Thus, the product of subgroups satisfies:
\[
H \cdot H = \{ xy \mid x, y \in H \} = H
\]
Since $H \otimes H \cong H \cdot H$, we have $[H]^2 = [H]$. Thus, every element in the image of $\Theta$ is idempotent.

Next, we check that every idempotent in $\Pic(\Spec\,\Z)$ is in the image of $\Theta$.

Let $[H] \in \Pic(\Spec\,\Z)$ be an idempotent element. Let $D=(n_p)$ be a generalized divisor representing $H$.
The condition $[H]^2 = [H]$ implies that $H \otimes H \cong H$. In terms of generalized divisors, this means:
$
2D \sim D
$.
By definition of linear equivalence, there exists a rational number $q \in \Q^\times$ such that:
\[
2D = D + \operatorname{div}(q)
\]
Looking at the coefficient for each prime $p$:
\[
2n_p = n_p + v_p(q)
\]
We analyze this equation for each $p$:
\begin{itemize}[leftmargin=*, labelindent=0pt]
    \item[] \emph{Case $n_p = \infty$.} The equation $2(\infty) = \infty + v_p(q)$ is satisfied for any finite $v_p(q)$.
    \item[] \emph{Case $n_p \in \Z$.} We can subtract $n_p$ from both sides (since it is finite), yielding:
    \[
    n_p = v_p(q)
    \]
\end{itemize}
This implies that the divisor $D$ splits into an infinite part and a principal part:
\[
D = \sum_{n_p = \infty} \infty \cdot [p] + \sum_{n_p \in \Z} v_p(q) \cdot [p]
\]
\[
D = D_\infty + \operatorname{div}(q)
\]
where $D_\infty$ corresponds to the set $S = \{ p \mid n_p = \infty \}$.
Thus $D$ is linearly equivalent to the divisor
\[
D_\infty=\sum_{p\in S}\infty\cdot[p],
\qquad
S=\{p\mid n_p=\infty\},
\]
and the corresponding subgroup is $\Z_S$.
Hence $[H]=\Theta(S)$.
\end{proof}

The set of idempotents in a commutative monoid naturally forms a join-semilattice with the order $e \le f \iff ef = f$. Keeping in mind that:
\begin{itemize}[leftmargin=*, labelindent=0pt]
    \item[] In $\mathcal{P}(\Spec \Z)$, the operation is the union $S \cup T$.
    \item[] In  $\Idem(\Pic(\Spec\,\Z))$, the operation is the tensor product $H_1 \otimes H_2$,
\end{itemize}

it follows that the map $\Theta$ is an isomorphism of semilattices:
\[
\Theta(S \cup T) = \Z_{S \cup T} \cong \Z_S \otimes_\Z \Z_T = \Theta(S) \cdot \Theta(T).
\]

\begin{rem} The range of the  map $\Theta: \mathcal{P}(\spz) \to  \Pic(\spz)$ is a subset of $\Pic(\spz)$ whose effective cardinality is that of the continuum, and hence is strictly smaller than the effective cardinality of $\Pic(\spz)$ as explained in Remark \ref{choice}.
\end{rem}


\section{The arithmetic Picard monoid of $\spzb$}\label{extended}

We extend the divisor-theoretic and Picard-theoretic constructions
developed for $\spz$ to its completion $\spzb$, where the archimedean place is taken
into account.
The guiding principle is to give an intrinsic geometric description of the \emph{arithmetic
Picard monoid} that does not rely a priori on adeles, rather it generalizes the
notion of a metrized line bundle familiar from Arakelov geometry.
We introduce arithmetic divisors on $\spzb$ as rank-1 torsion-free groups equipped
with a (possibly degenerate) metric, define the corresponding arithmetic Picard
monoid, and then relate this intrinsic picture to an adelic description.
A key outcome is the construction of a canonical invariant subgroup of $\R$ attached
to each arithmetic divisor, which allows one to reconstruct the divisor from its
value spectrum when the semi-norm is a norm.
We describe the natural commutative monoid structure on $\Pic(\spzb)$ and
show its compatibility at the intrinsic, adelic, and spectral levels.

\subsection{Arithmetic divisors on $\spzb$}

We begin by defining arithmetic divisors on the completed curve $\spzb$.
These objects should be viewed as the natural analogues of metrized line bundles in
Arakelov geometry, with the important distinction that we allow degenerate metrics.

\begin{defn}
An arithmetic divisor on $\spzb$ is a pair $\mathcal{D} = (L, \|\cdot\|)$ where:
\begin{itemize}
    \item $L$ is a torsion-free abelian group of rank 1.
    \item $\|\cdot\|$ is a semi-norm on the real vector space $L_\R = L \otimes_\Z \R$ which is proportional to the standard absolute value $|\cdot|$ on $\R$ (via any identification $L_\R \cong \R$) 
\end{itemize}
\end{defn}

The condition on the semi-norm means that there exists an isomorphism $\phi: L_\R \xrightarrow{\sim} \R$ and a constant $\lambda \ge 0$ such that  $\|v\| = \lambda |\phi(v)|$ for all $v \in L_\R$.
\begin{itemize}[leftmargin=*, labelindent=0pt]
    \item[] If $\lambda > 0$, the metric is non-degenerate (Hermitian).
    \item[] If $\lambda = 0$, the metric is singular (the zero semi-norm).
\end{itemize}

\subsection{Intrinsic product}

Let $\mathcal{D}_1 = (L_1, \|\cdot\|_1)$ and $\mathcal{D}_2 = (L_2, \|\cdot\|_2)$ be two arithmetic divisors on $\spzb$. We define their product $\mathcal{D}_1 \cdot \mathcal{D}_2$ as the pair $(L, \|\cdot\|)$ where:
\begin{itemize}[leftmargin=*, labelindent=0pt]
    \item[-] The group $L$ is the tensor product: $L = L_1 \otimes_\Z L_2$.
    \item[-] The semi-norm $\|\cdot\|$ on $L_\R \cong (L_1)_\R \otimes (L_2)_\R$ is the tensor product of the semi-norms. Explicitly, for pure tensors $v_1 \otimes v_2$, we have:
    \[
    \|v_1 \otimes v_2\| = \|v_1\|_1 \cdot \|v_2\|_2.
    \]
\end{itemize}
Since $L_1$, and  $L_2$ are rank-1 torsion-free groups, their tensor product is also a rank-1 torsion-free group. The product of their semi-norms is well-defined and proportional to the standard absolute value on $\R$.

\subsection{The  Picard monoid $\Pic(\spzb)$}

Having defined arithmetic divisors on $\spzb$, and their product, we now pass to their isomorphism
classes.
This leads to the arithmetic Picard monoid as the moduli space of these objects, which plays the role of the Picard group of
the completed curve.

\begin{defn}
Two arithmetic divisors $\mathcal D_1=(L_1, \|\cdot\|_1)$ and $\mathcal D_2=(L_2, \|\cdot\|_2)$ on $\spzb$ are isomorphic if there exists a group isomorphism $\psi: L_1 \to L_2$ that is an isometry, \ie  for all $x \in L_1$, one has:
\[
\|\psi(x)\|_2 = \|x\|_1.
\]
Here, we implicitly extend $\psi$ to the real vector spaces.
\end{defn}

\begin{defn}
The arithmetic Picard monoid $\Pic(\spzb)$  is the set of isomorphism classes of arithmetic divisors on $\spzb$ endowed with the product $\mathcal{D}_1 \cdot \mathcal{D}_2$.
\end{defn}

\subsection{Adelic description}\label{adrel} 

We next relate the intrinsic definition of $\Pic(\spzb)$ to an adelic
parametrization.
This comparison clarifies how the finite and archimedean components jointly encode
the arithmetic data of a divisor class.\vspace{.03in}

      Let $\Lambda$ be the map which associates to a rational adele $a=(a_f,a_\infty)$ the pair consisting of:
\begin{itemize}[leftmargin=*, labelindent=0pt]
    \item The group $L_a := \{q \in \Q \mid a_f q \in \widehat{\Z}\}$,
    \item The semi-norm $\|\cdot\|_a$ on $L_a \otimes_\Z \R$ defined by $\|v\|_a = |a_\infty| \cdot |v|$ (via the canonical identification $L_a \otimes \R \cong\Q\otimes \R\cong \R$).
\end{itemize}
We denote with $X_\Q=\Q^\times \backslash \A / \Zhat^\times$  the \emph{Riemann sector} of the adele class space $\A$ of the rationals, endowed with the monoid structure induced by the multiplication of adeles. We now prove Theorem \ref{xqthm}.

\begin{thm} \label{xqthm1} 
The map $\Lambda$ induces a canonical monoid isomorphism $X_\Q \cong \Pic(\spzb)$.
\end{thm} 

\begin{proof} 
We proceed in four steps.

\emph{1. Well-definedness.} 
We show that given an adele $a$,  $a\mapsto \Lambda(a)$ depends only on the class of $a$ in $X_\Q$.
First, let $u \in \Zhat^\times$. Since units do not affect integrality, we have $L_{ua} = L_a$. As $(ua)_\infty = a_\infty$, the norms coincide, so $\Lambda(ua) = \Lambda(a)$.
Next, let $r \in \Q^\times$ and consider $a' = ra$. The group $L_{a'}$ is defined by the condition $q(ra)_f \in \Zhat$, which is equivalent to $qr \in L_a$. Thus $L_{a'} = r^{-1} L_a$.
The archimedean component transforms as $a'_\infty = r a_\infty$. Consider the multiplication map $\psi: L_{a'} \to L_a$ defined by $\psi(q) = rq$. This is a group isomorphism satisfying:
\[
\|\psi(q)\|_a = |a_\infty| \, |rq| = |r a_\infty| \, |q| = |a'_\infty| \, |q| = \|q\|_{a'}.
\]
Thus, $\psi$ defines an isomorphism of arithmetic divisors $\Lambda(a') \cong \Lambda(a)$. This confirms that $\Lambda$ descends to the quotient $X_\Q$.

\emph{2. Injectivity.} 
Let $a, b \in \A$ be such that $\Lambda(a)\cong \Lambda(b)$. Let $\psi: L_a \to L_b$ be an isomorphism compatible with the norms. Since $L_a$ and $L_b$ are rank-1 subgroups of $\Q$, there exists a unique $r \in \Q^\times$ such that $\psi(x) = rx$ for all $x \in L_a$.
Define $b' = rb$. As shown in step 1, the map $\phi(y) = r^{-1}y$ provides an isomorphism $\Lambda(b) \cong \Lambda(b')$. Consequently, the composition $\rho = \phi \circ \psi$ is an automorphism of $\Lambda(a)$ to $\Lambda(b')$ given by $\rho(x) = x$.
The equality of the underlying groups $L_a = L_{b'}$ implies that the finite adeles $a_f$ and $b'_f$ generate the same lattice, so there exists $u \in \Zhat^\times$ such that $a_f = u b'_f$.
The equality of the norms $\|x\|_a = \|x\|_{b'}$ implies $|a_\infty| |x| = |b'_\infty| |x|$, hence $|a_\infty| = |b'_\infty|$. Thus $a_\infty = \epsilon b'_\infty$ for some $\epsilon \in \{\pm 1\}$.
It follows that $a = (u, \epsilon) b' = (u, \epsilon) r b$, so $a$ and $b$ represent the same class in $X_\Q$.

\emph{3. Surjectivity.} 
This follows from Proposition \ref{ident}. Given any arithmetic divisor $(L, \|\cdot\|)$ with $L \subset \Q$, there exists a finite adele $a_f$ such that $L = L_a$. The norm on $L \otimes \R \cong \R$ is necessarily of the form $\|x\| = \lambda |x|$ for some $\lambda \ge 0$. Choosing $a_\infty$ such that $|a_\infty| = \lambda$, we obtain $\Lambda(a) \cong (L, \|\cdot\|)$.

\emph{4. Monoid Structure.} 
Proposition \ref{ident} establishes that the map $a \mapsto L_a$ is a monoid homomorphism, identifying the product of adeles with the product of subgroups (which is isomorphic to the tensor product). For the norms, we observe:
\[
\|x \otimes y\|_{ab} = |(ab)_\infty| |xy| = (|a_\infty| |x|) (|b_\infty| |y|) = \|x\|_a \|y\|_b.
\]
Thus, $\Lambda$ preserves the tensor product structure.
\end{proof}

\subsection{The adelic invariant subgroup $I_a$}

Let $a=(a_f,a_\infty) \in \A$ be a rational adele representing a divisor class in $\Pic(\spzb)$ (in view of Theorem \ref{xqthm1}). While the group $L_a \subset \Q$ defined by the finite part $a_f$ of the adele depends on the representative $a\in\A$, we shall construct a \emph{canonical realization} of this group inside $\R$ that serves as a complete invariant for the class.

\begin{defn}\label{adinv}
Given a rational adele $a\in\A$, we define the subgroup $I_a \subset \R$ as the image of the group $L_a\subset \Q$ scaled by the archimedean component $a_\infty$
\[
I_a := a_\infty L_a = \{ x \in \R \mid x = a_\infty q, \text{ and } q a_f \in \Zhat \}.
\]
\end{defn}

\begin{thm}[Invariance]
The subgroup $I_a\subset \R$ depends only on the class of $a\in\A$ in the double quotient $\Q^\times \backslash \A / \Zhat^\times$.
\end{thm}

\begin{proof}
This follows directly from the calculations in the proof of Theorem \ref{xqthm1}. Invariance under $\Zhat^\times$ is immediate as units do not affect the lattice $L_a$. For invariance under $\Q^\times$, if $a' = r a$ with $r \in \Q^\times$, we established that $L_{a'} = r^{-1} L_a$ and $a'_\infty = r a_\infty$. Consequently, $I_{a'} = (r a_\infty)(r^{-1} L_a) = a_\infty L_a = I_a$.
Alternatively, this invariance is a consequence of the intrinsic characterization of $I_a$ via the value spectrum given in Lemma \ref{lem:spectrum} below.
\end{proof}

Next, we reformulate the adelic invariant subgroup in intrinsic terms, using only the metric
structure of an arithmetic divisor $\mathcal{D} = (L, \|\cdot\|)\in \Pic(\spzb)$.
This leads to the notion of the value spectrum and allows for a reconstruction
theorem. 

\begin{defn}
The value spectrum of an arithmetic divisor $\mathcal{D} = (L, \|\cdot\|)\in \Pic(\spzb)$ is the image of the group $L$ under the following norm map:
\[
S(\mathcal{D}) := \{ \|x\| \mid x \in L \} \subset [0, \infty).
\]
\end{defn}

Let $I(\mathcal{D}) \subset \R$ be the subgroup generated by the set $S(\mathcal{D})$.
If $\mathcal{D}$ is represented by an adele $a\in\A$, then $I(\mathcal{D})$ is precisely the invariant subgroup $I_a = a_\infty L_a$ defined previously.

\begin{lem}\label{lem:spectrum}
Let $a \in \A$ be an adele representing a class $\mathcal{D}$ in $\Pic(\spzb)$. Then the value spectrum $S(\mathcal{D})$ coincides with the set of non-negative elements of the invariant subgroup $I_a$:
\[
S(\mathcal{D}) = I_a \cap [0, \infty).
\]
Conversely, since the subgroup $I_a \subset \R$ is symmetric with respect to the origin, it is fully recovered from the spectrum:
\[
I_a = S(\mathcal{D}) \cup (-S(\mathcal{D})).
\]
\end{lem}

\begin{proof}
Let $\mathcal D=(L,\|\cdot\|)$ be the class represented by $a=(a_f,a_\infty)\in\A$.
By construction (cf.\ Definition~\ref{adinv}), the finite component $a_f$ determines a
rank-$1$ subgroup
\[
L_a:=\{q\in\Q \mid q\,a_f\in\widehat\Z\}\subset \Q,
\]
and the associated invariant subgroup of $\R$ is
\[
I_a=a_\infty L_a\subset \R.
\]
Under the identification of $L_\R$ with $\R$ coming from the chosen representative
$a$ (i.e.\ the embedding $L\hookrightarrow \Q$ and the scaling factor $a_\infty$),
the norm is given by
\[
\|q\|=|a_\infty q|\qquad (q\in L_a).
\]
Therefore the value spectrum of $\mathcal D$ is
\[
S(\mathcal D)=\{\|q\|\mid q\in L_a\}
=\{|a_\infty q|\mid q\in L_a\}
=\{|x|\mid x\in I_a\}.
\]
Since $I_a$ is a subgroup of $\R$, it is stable under $x\mapsto -x$, hence
\[
\{|x|\mid x\in I_a\}=I_a\cap[0,\infty),
\]
which proves the first identity.

For the second statement, the symmetry $I_a=-I_a$ implies that every element
$x\in I_a$ is either in $I_a\cap[0,\infty)$ or in its negative. Using the first part,
this yields
\[
I_a=(I_a\cap[0,\infty))\cup -(I_a\cap[0,\infty))
=S(\mathcal D)\cup(-S(\mathcal D)).
\]
\end{proof}

The preceding lemma shows that an arithmetic divisor determines a \emph{canonical} value
spectrum inside $\R$. Next we  show that
this invariant is complete, provided the metric is not identically zero.

\begin{defn}[Non-degenerate divisor]
An arithmetic divisor $\mathcal D=(L,\|\cdot\|)$ on $\spzb$ is  \emph{non-degenerate} if the semi-norm $\|\cdot\|$ is not identically zero.
\end{defn}

\begin{rem}
If $\mathcal D$ is represented by an adele $a \in \A$, non-degeneracy corresponds to the condition $a_\infty \neq 0$. In this case, the map $L \to \R$ induced by the norm is injective, and $L$ is isomorphic to the subgroup $I(\mathcal D) \subset \R$.
\end{rem}

\begin{thm}[Reconstruction]\label{rec}
Let $\mathcal{D}_1, \mathcal{D}_2$ be two non-degenerate arithmetic divisors.
Then,  $\mathcal{D}_1$ and $\mathcal{D}_2$ are isomorphic if and only if their value spectra are identical:
\[
[\mathcal{D}_1] = [\mathcal{D}_2] \iff S(\mathcal{D}_1) = S(\mathcal{D}_2).
\]
\end{thm}

\begin{proof}
``$\Rightarrow$'' Isomorphisms are isometries, thus they preserve the set of norms.

``$\Leftarrow$'' Suppose $S(\mathcal{D}_1) = S(\mathcal{D}_2)$.
By Lemma \ref{lem:spectrum}, the invariant subgroups generated by the spectra coincide:
\[
I(\mathcal{D}_1) := S(\mathcal{D}_1) \cup -S(\mathcal{D}_1) = S(\mathcal{D}_2) \cup -S(\mathcal{D}_2) = I(\mathcal{D}_2).
\]
Let $I \subset \R$ denote this common subgroup. Since the divisors are non-degenerate, $I \neq \{0\}$.
For $k=1,2$, the divisor $\mathcal{D}_k = (L_k, \|\cdot\|_k)$ is canonically isomorphic to the divisor $(I, |\cdot|_{\text{std}})$ defined by the subgroup $I$ with the standard absolute value induced from $\R$.
Indeed, the norm $\|\cdot\|_k$ induces an embedding $\iota_k: L_k \hookrightarrow \R$ whose image is $I$. The map $\iota_k$ constitutes an isomorphism of arithmetic divisors:
\[
\iota_k: (L_k, \|\cdot\|_k) \xrightarrow{\sim} (I, |\cdot|_{\text{std}}).
\]
Consequently, we have the isomorphism $\iota_2^{-1} \circ \iota_1: \mathcal{D}_1 \xrightarrow{\sim} \mathcal{D}_2$.
\end{proof}

\subsection{Product of value spectra}

Finally, we describe the product operation in terms of the invariant subgroups of $\R$.
Let $I_1, I_2 \subset \R$ be the invariant subgroups associated to two arithmetic divisors on $\Pic(\spzb)$. The invariant subgroup of the product is the subgroup generated by the product of the sets:
\[
I_1 \cdot I_2 := \langle xy \mid x \in I_1, y \in I_2 \rangle.
\]
In terms of the value spectra $S_i = I_i \cap [0, \infty)$, the spectrum of the product is the closure of the set of products under addition:
\[
S(\mathcal{D}_1 \cdot \mathcal{D}_2) = \{ |z| \mid z \in I_1 \cdot I_2 \}.
\]
This confirms that the \emph{reconstruction map} 
\[
\mathcal{D} \mapsto I(\mathcal{D})= S(\mathcal{D}) \cup -S(\mathcal{D})
\]
is a monoid homomorphism from the arithmetic Picard $\Pic(\spzb)$ to the monoid of subgroups of $\R$ (under the product operation).


\section{The arithmetic Jacobian $\Jac(\spzb)$}\label{jacob}

We introduce the arithmetic Jacobian of the completed curve
$\spzb$.
While the arithmetic Picard monoid $\Pic(\spzb)$ parametrizes rank-$1$
torsion-free abelian groups endowed with a (possibly degenerate) metric at the
archimedean place, the Jacobian is obtained by identifying divisors that differ only
by a rescaling of this metric.
From a geometric perspective, this passage amounts to obtain the Jacobian by quotienting the Picard group by the subgroup generated by the  divisor of degree $1$ used to define the Abel-Jacobi map. In our set-up this  is the subgroup $\R_+^\times\subset \Pic(\spzb)$ and  the quotient monoid $\Jac(\spzb)$ is the space
of orbits of the archimedean scaling flow.
We describe the resulting structure explicitly, endow it with a natural monoid law,
and explain how the arithmetic Abel--Jacobi map extends to a map from $\spzb$ to  $\Jac(\spzb)$.

\subsection{Definition of $\Jac(\spzb)$}

The multiplicative group of positive real numbers $\R_+^\times$ acts naturally on the
arithmetic Picard monoid $\Pic(\spzb)$ by rescaling the metric at infinity.
Explicitly, for an arithmetic divisor $\mathcal D=(L,\|\cdot\|)$ and $\lambda\in\R_+^\times$
one sets
\[
\lambda\cdot(L,\|\cdot\|):=(L,\lambda\|\cdot\|).
\]
Under the adelic identification, this action corresponds to multiplication of the
archimedean component $a_\infty$ by $\lambda$.

\begin{defn}
The \emph{arithmetic Jacobian} of $\spzb$ is  the quotient
\[
\Jac(\spzb)=\Pic(\spzb)\big/\R_+^\times.
\]
\end{defn}
This construction fits into an exact sequence of commutative monoids
$$
1\to \R_+^\times\to \Pic(\spzb)\to \Jac(\spzb)\to 1
$$
which is the arithmetic analogue of the classical decomposition of the Picard group
of a complete curve over a finite field into its Jacobian part and a degree
component.

\subsection{Structure of $\Jac(\spzb)$}

The structure of  $\Jac(\spzb)$ is governed by the behavior of the metric under the
scaling action.
Two qualitatively distinct cases arise.
\begin{enumerate}[leftmargin=*, labelindent=0pt]
\item \emph{Finite type (flow lines).}
If the metric $\|\cdot\|$ is non-zero, its $\R_+^\times$-orbit consists of all non-zero
metrics on the real line $L_\R$.
After passing to the quotient, the resulting class is determined solely by the
isomorphism class of the underlying group $L$.
We denote such an element simply by $L$.

\item \emph{Infinite type (fixed points).}
If the metric is the zero semi-norm, it is fixed by the $\R_+^\times$-action.
The corresponding class records both the group $L$ and the degeneracy of the metric.
We denote such an element by $L_\infty$.
\end{enumerate}

As a consequence, one obtains a canonical identification
\[
\Jac(\spzb)\;\cong\; \Pic(\spz)\times\{0,\infty\},
\]
where $\Pic(\spz)$ is the set of isomorphism classes of torsion-free rank-$1$
abelian groups.
The second factor records the archimedean valuation: $0$ for non-degenerate metrics
and $\infty$ for the degenerate (zero) metric.

\subsection{Monoid structure}

The commutative monoid structure on $\Pic(\spzb)$ induced by tensor product
descends naturally to $\Jac(\spzb)$.
On the group component it is given by tensor product, while the archimedean component $n_\infty\in\{0,\infty\}$
follows the logic of an absorbing value:
\begin{itemize}[leftmargin=*, labelindent=0pt]
\item[-] $H\cdot K = H\otimes_\Z K$ \hfill ($0+0=0$)
\item[-] $H\cdot K_\infty = (H\otimes_\Z K)_\infty$ \hfill ($0+\infty=\infty$)
\item[-] $H_\infty\cdot K_\infty = (H\otimes_\Z K)_\infty$ \hfill ($\infty+\infty=\infty$).
\end{itemize}

Thus $\Jac(\spzb)$ inherits a natural commutative monoid structure compatible with
its interpretation as a quotient of the arithmetic Picard monoid.

\subsection{The  extended Abel-Jacobi map}\label{sheaf}

The arithmetic Abel--Jacobi map  constructed earlier (Definition~\ref{aAJ}) extends to the Jacobian
$\Jac(\spzb)$.
From the adelic point of view, this map assigns to each point of the completed curve
$\spzb$ the orbit it generates under the archimedean scaling action.
In this way, points of $\spzb$ are classified by their reduced arithmetic
invariants

\begin{table}[H]
\centering
{\fontsize{10.3}{12.3}\selectfont
\begin{tabular}{c|c|c|c}
\textbf{Point in $\spzb$} & \textbf{adelic representative} & \textbf{reduced element} & \textbf{geometric meaning} \\
\hline
generic point $\eta$ & $(1_f, 1_\infty)$ & $\Z$ & generic orbit (flow line) \\
\hline
$p<\infty$ & $(e_p, 1_\infty)$ & $\Z[1/p]$ & periodic orbit $C_p$ \\
\hline
$p=\infty$ & $(1_f, 0_\infty)$ & $\Z_\infty$ & archimedean singularity \\
\hline
\end{tabular}
}
\caption{The image of the curve in the arithmetic Jacobian}\label{table}
\end{table}
In formulas one writes:
\begin{equation}\label{extAJ}
\Theta: \spzb \longrightarrow \Jac(\spzb), \quad\Theta(x)= \begin{cases}(\mathbf{Z}[1 / p],\|\cdot\|\neq 0) & \text { if } x=p<\infty  \\ (\mathbf{Z}, 0)=\Z_\infty & \text { if } x=\infty \\ (\mathbf{Z},\|\cdot\|\neq 0) & \text { if } x=\eta.\end{cases}
\end{equation}


\section{Sheaves of modules on $\spzb$}\label{sectsheaf}

In this section we relate the arithmetic Picard monoid $\Pic(\spzb)$ to a
sheaf-theoretic framework inspired by the theory of $\mathbf F_1$-algebras developed
in \cite{CC2}.
Our goal is to show that every arithmetic divisor on the completed curve
$\spzb$ naturally determines a sheaf of modules over the structure sheaf of $\spzb$.
This construction extends the classical correspondence between divisors and
invertible sheaves, while simultaneously encoding both the finite arithmetic data
and the archimedean metric constraint.

\subsection{The Structure sheaf}

We regard $\spzb$ as a topological space whose open sets are complements of
finite sets of closed points (finite primes and possibly the point at infinity).
The structure sheaf $\mathcal O_{\spzb}$ corresponds to the trivial arithmetic
divisor
\[
\mathbf{1}=(\Z,|\cdot|_\infty),
\]
where $|\cdot|_\infty$ denotes the standard absolute value on $\R$.

For an open subset $U\subset\spzb$, the algebra of sections is defined as a
sub-$\mathbf F_1$-algebra of $H\Q$ as follows.

\begin{itemize}[leftmargin=*, labelindent=0pt]
\item If $\infty\notin U$, then
\[
\mathcal O_{\spzb}(U)=H\Z_U,
\]
where $\Z_U=\Z[S^{-1}]$ is the localization of $\Z$ where the set of primes is
$S=\Spec\Z\setminus U$.

\item If $\infty\in U$, we impose the archimedean metric condition:
\[
\mathcal O_{\spzb}(U)(X)
=
\left\{
\phi\in H\Z_{U\setminus\{\infty\}}(X)
\;\middle|\;
\sum_{x\in X\setminus\{*\}}|\phi(x)|_\infty\le 1
\right\}.
\]
\end{itemize}

This definition recovers the structure sheaf introduced in
\cite{CC2}, Proposition~6.3.

\subsection{Sheaves associated to arithmetic divisors}

We now extend the above construction  to an arbitrary
arithmetic divisor
\[
\mathcal D=(L,\|\cdot\|)\in\Pic(\spzb).
\]
The finite part of the data is encoded in the subgroup $L\subset\Q$, while the
archimedean contribution is governed by the semi-norm $\|\cdot\|$.

We first associate to $L$ a sheaf of abelian groups $\mathcal L$ on the Zariski
topology of $\Spec\Z$.
For an open subset $V\subset\Spec\Z$, the group $\mathcal L(V)$ consists of those
elements of $L$ that are regular on $V$, i.e.\ whose denominators involve only
primes in the complement $V^c$.

\begin{prop}
Let $\mathcal D=(L,\|\cdot\|)$ be an arithmetic divisor on $\spzb$.
The following construction defines a sheaf
$\mathcal O(\mathcal D)$ of $\mathcal O_{\spzb}$-modules on $\spzb$.
\begin{enumerate}
\item(Finite part)
If an open set $U$ does not contain $\infty$, set
\[
\mathcal O(\mathcal D)(U):=H\mathcal L(U).
\]

\item(Infinite part)
If $U$ contains $\infty$, define
\[
\mathcal O(\mathcal D)(U)(X)
:=
\left\{
\phi\in H\mathcal L(U\cap\Spec\Z)(X)
\;\middle|\;
\sum_{x\in X\setminus\{*\}}\|\phi(x)\|\le 1
\right\}.
\]
\end{enumerate}
\end{prop}

\begin{proof}
The functor $H\mathcal L(U\cap\Spec\Z)$ provides the underlying $\Gamma$-set
structure.
It remains to check that the metric condition at infinity is compatible with the
module structure over $\mathcal O_{\spzb}$.

Let $\psi\in\mathcal O_{\spzb}(U)(Y)$ and
$\phi\in\mathcal O(\mathcal D)(U)(X)$.
Their smash product acts on $z=(y,x)$ by multiplication.
By definition of the semi-norm,
\[
\|\psi(y)\phi(x)\|=|\psi(y)|_\infty\,\|\phi(x)\|.
\]
Summing over the smash product yields
\[
\sum_{y,x}\|\psi(y)\phi(x)\|
=
\left(\sum_y|\psi(y)|_\infty\right)
\left(\sum_x\|\phi(x)\|\right)
\le 1\cdot 1=1.
\]
Thus the metric constraint is preserved, and $\mathcal O(\mathcal D)$ defines a
sheaf of $\mathcal O_{\spzb}$-modules.
\end{proof}

\begin{rem}
This construction uniformly accounts for both finite and infinite types in the
arithmetic Picard. Precisely:\vspace{.03in}

(Finite type)
If the semi-norm $\|\cdot\|$ is non-degenerate, the condition
$\sum\|\phi(x)\|\le 1$ defines a unit ball of sections, generalizing the classical
Arakelov metric condition.\vspace{.03in}

(Infinite type)
If $\|\cdot\|=0$, the inequality $\sum\|\phi(x)\|\le 1$ is automatically satisfied.
Hence for a divisor of infinite type $H_\infty$, the sheaf
$\mathcal O(H_\infty)$ imposes no constraint at infinity, corresponding
geometrically to allowing a pole of arbitrary order at $\infty$.

\end{rem}


\section{Arithmetic pullbacks and the geometry of the Riemann sector}\label{sectpullback}

This section develops the arithmetic analogue of a classical geometric construction:
the pullback of an abelian étale covering of a Jacobian along the Abel--Jacobi map.
Our aim is to interpret the ``visible part'' $\mathcal X_\Q$ of the Riemann sector $X_\Q$ as a geometric covering of the
arithmetic curve $\spzb$, and to describe its fibers in terms of arithmetic divisors.

The guiding principle is that the arithmetic Picard monoid $\Pic(\spzb)$ plays the
role of a  covering space of $\Jac(\spzb)$, while the extended Abel--Jacobi map determines
a pullback whose fibers encode the dynamical and group--theoretic features of the
scaling action.  

We begin by recalling the classical geometric prototype over a finite
field, and then transpose it to the arithmetic setting.  The remainder of the section
is devoted to a detailed analysis of the resulting fibers.

\subsection{The classical geometric prototype}

We briefly recall the standard construction of abelian unramified covers of a curve
over a finite field, which will serve as a geometric template for the arithmetic case.\vspace{.03in}

Let $C$ be a smooth projective curve of genus $g$ over a finite field $k$. Let $J = \operatorname{Jac}(C)$ be its Jacobian variety.
Class field theory for the function field $k(C)$ establishes a correspondence between abelian unramified extensions of $k(C)$ and isogenies of the Jacobian $J$.

Fix a prime $\ell$ distinct from the characteristic of $k$. For any integer $n \ge 1$, consider the multiplication by $\ell^n$ map on the Jacobian:
\[
[\ell^n]: J \longrightarrow J, \quad x \longmapsto \ell^n x.
\]
This is a finite étale covering of degree $\ell^{2gn}$, whose Galois group is the group
of $\ell^n$--torsion points $J[\ell^n](k)$, acting by translations:
\[
\tau \cdot x = x + \tau, \quad \text{for } \tau \in J[\ell^n],~ x \in J.
\]
To obtain a covering of the curve $C$, one pulls back this covering of the Jacobian
along the Abel--Jacobi embedding $\alpha : C \hookrightarrow J$ defined with respect
to a base point $P$:
\[
\begin{tikzcd}
C_n \arrow[r] \arrow[d] & J \arrow[d, "{\times \ell^n}"] \\
C \arrow[r, "\alpha"] & J
\end{tikzcd}
\]
The resulting curve $C_n$ is an étale cover of $C$, which realizes the maximal abelian
unramified covering of exponent $\ell^n$.

\subsection{The arithmetic pullback}\label{arpull}

We now transpose the preceding construction to the arithmetic curve $\spzb$.
Each ingredient of the geometric picture admits a natural arithmetic counterpart:
\begin{itemize}[leftmargin=*, labelindent=0pt]
  \item[-] the Jacobian $J$ is replaced by the arithmetic Jacobian $\Jac(\spzb)$;
  \item[-] the isogeny $[\ell^n] : J \to J$ is replaced by the projection
        $\pi : \Pic(\spzb) \to \Jac(\spzb)$;
  \item[-] The role of the Galois group of translations $J[\ell^n]$ is replaced by the continuous scaling
        group $\R_+^\times$, the connected component of the idele class group;
  \item[-] translations on the Jacobian are given by \emph{multiplicative scaling} of
        arithmetic divisors.
\end{itemize}

Pulling back the adelic covering along the extended Abel--Jacobi map
$\Theta$~\eqref{extAJ}, we obtain the diagram:
\begin{equation}\label{2tow}
\begin{tikzcd}
\mathcal{X}_\Q \arrow[r] \arrow[d] & \Pic(\spzb) \arrow[d, "\pi"] \\
\spzb \arrow[r, "\Theta"] & \Jac(\spzb)
\end{tikzcd}
\end{equation}
The resulting space $\mathcal{X}_\Q$ will be referred to as the ``\emph{visible part}''
of the Riemann sector $X_\Q$, now realized geometrically as a covering of the
arithmetic curve.  Its fibers encode the periodic orbits and fixed points of the
arithmetic flow, which we analyze next.

\subsection{Group structure on the periodic orbits}

We first study the fibers lying over closed points of $\spzb$, which correspond
to periodic orbits of the scaling action.

Let $p < \infty$ be a finite prime, viewed as a closed point of $\spzb$, and let
\[
H_p = \Z[1/p]
\]
be its image in the Jacobian under the Abel--Jacobi map $\Theta$ \eqref{extAJ}.  The fiber of $\pi$ (as in \eqref{2tow}) over
$H_p$, denoted $C_p$, consists of all arithmetic divisors whose underlying group is
isomorphic to $\Z[1/p]$ and whose metric is non--zero.  By 
Theorem~\ref{rec}, this fiber is identified with
\[
C_p = \{\, I \subset \R \mid I \cong \Z[1/p] \text{ as abelian groups} \,\}.
\]

\subsubsection{Idempotence and the induced group law}

The monoid structure of the arithmetic Picard induces a binary operation on each fiber $C_p$.
On subgroups of $\R$, this operation is given by
\[
I_1 \cdot I_2 = \langle xy \mid x \in I_1,\ y \in I_2 \rangle .
\]

\begin{prop}
The set $C_p$ is an abelian group with respect to this operation.
\end{prop}

\begin{proof}
Let $I_1, I_2 \in C_p$. Since $I_1 \cong I_2 \cong \Z[1/p]$, both are modules over the ring $\Z[1/p]$. Their product $I_1 \cdot I_2$ is the image of the tensor product $I_1 \otimes_\Z I_2$ inside $\R$.
Using the idempotence of the localization $\Z[1/p] \otimes_\Z \Z[1/p] \cong \Z[1/p]$, we see that $I_1 \cdot I_2$ is again a rank-1 module over $\Z[1/p]$, hence isomorphic to $\Z[1/p]$. Thus $I_1 \cdot I_2 \in C_p$.

Consider the canonical embedding $I_0 = \Z[1/p] \subset \R$ (corresponding to $1\mapsto 1$).  Since $I$ is a $\Z[1/p]$-module, for any $I \in C_p$ we have:
\[
I_0 \cdot I = \{ \lambda x \mid \lambda \in \Z[1/p], x \in I \} = I.
\]
Thus $I_0$ is the neutral element.

Any $I \in C_p$ is of the form $I = \lambda I_0$ for some $\lambda \in \R^\times$.
Let $I^{-1} = \lambda^{-1} I_0$. Then:
\[
I \cdot I^{-1} = (\lambda I_0) \cdot (\lambda^{-1} I_0) = (\lambda \lambda^{-1}) (I_0 \cdot I_0) = 1 \cdot I_0 = I_0.
\]
Thus every element is invertible.
\end{proof}

\subsubsection{Identification with the circle group}

The parametrization $I_\lambda = \lambda \Z[1/p]$ defines a surjective homomorphism
$\R^\times \to C_p$.  Its kernel is the stabilizer of the neutral element $I_0$:
\[
\operatorname{Stab}(I_0) = \{ \lambda \in \R^\times \mid \lambda \Z[1/p] = \Z[1/p] \} = (\Z[1/p])^\times = \{ \pm p^k \mid k \in \Z \}.
\]
Restricting to positive scalars (that corresponds to the connected component of the identity) yields a canonical isomorphism of topological groups
\[
C_p \cong \R_+^\times / p^\Z \cong \R / (\log p)\Z ,
\]
so that each fiber $C_p$ is canonically a circle group, and the neutral element of the fiber is the specific point on the orbit corresponding to the standard embedding of the rational numbers.

\subsection{The fiber at infinity}

We now consider the fiber over the archimedean point $\infty \in \spzb$.
Its image under $\Theta$ is the reduced element $\Z_\infty$, corresponding to the
group $\Z$ equipped with the zero semi--norm.

Accordingly, the fiber $C_\infty$ consists of arithmetic divisors with underlying
group $\Z$ and vanishing metric.  There is a unique such divisor:
\[
C_\infty = \{ (\Z, 0) \}.
\]
Geometrically, this point corresponds to the unique \emph{fixed point of the scaling flow}.

\subsubsection{Absorption by the fixed point}

Let $I_\infty = \{0\} \subset \R$ denote the subgroup corresponding to the point
at infinity.  For any $p < \infty$ and any $I_p \in C_p$, one has
\[
I_p \cdot I_\infty = \langle xy \mid x \in I_p, y \in \{0\} \rangle = \{0\} = I_\infty.
\]
Thus
\[
C_p \cdot C_\infty = C_\infty ,
\]
expressing the fact that the archimedean fiber is an absorbing element for the
monoid structure.  Dynamically, this reflects the role of $\infty$ as a sink for
the scaling flow (\ie the limit of the scaling action as $\lambda \to 0$ leads to the zero semi-norm).

\subsection{The generic fiber}

Finally, we consider the fiber over the generic point $\eta \in \spzb$.
We define its image under $\Theta$ to be the identity element of the Jacobian:
\[
\Theta(\eta) := [\Z] \in \Jac(\spzb) ,
\]
corresponding to the principal adele $(1_f,1_\infty)$.

\subsubsection{Description of the generic fiber}

The fiber $C_\eta$ consists of arithmetic divisors $(L,\|\cdot\|)$ with
$L \cong \Z$ and non--zero metric.  By  Theorem \ref{rec},
\[
C_\eta = \{\, I \subset \R \mid I \cong \Z \,\} = \{ \lambda \Z \mid \lambda \in \R^\times \}.
\]
Any such subgroup is of the form $I = \lambda \Z$ for some $\lambda \in \R^\times$.
Thus, we have an isomorphism of topological groups:
\[
C_\eta \cong \R_+^\times / \{1\} \cong \R_+^\times.
\]
Thus the fiber $C_\eta$ is simply the multiplicative group of positive real numbers (which may be viewed as a single flow line).

\subsubsection{Density and topological meaning}

The generic point $\eta$ is dense in the Zariski topology of $\spzb$.
Correspondingly, in the adelic topology of $\Pic(\spzb) \cong X_\Q$, the fiber
$C_\eta$ consists of principal arithmetic divisors associated to ideles.
By strong approximation, ideles are dense in the adeles, and therefore
\[
\overline{C_\eta} = \Pic(\spzb) .
\]
This identifies $C_\eta$ as the \emph{generic stratum} of the arithmetic Picard space,
linking algebraic density on $\spzb$ with analytic density in the adelic setting.


\section{The universal cover}\label{sectuniv}

We construct and analyze the universal covering objects naturally
associated to the adelic space $Y_\Q$.  The guiding idea is that finite abelian extensions
of $\Q$ admit a geometric incarnation as finite covers of $X_\Q$, in a way that
faithfully reflects the Galois correspondence of class field theory.
After introducing the class field functor and its basic properties, we analyze
the group structure carried by the fibers of the universal cover over the various
points of $\spzb$: finite primes, the archimedean place, and the generic point.

\subsection{The class field functor}\label{sectcft}

We begin by recalling how finite abelian extensions of $\Q$ give rise to geometric
covers of the adelic space $X_\Q$.  This construction realizes class field theory
directly at the level of topological spaces and dynamical systems.

Let $L$ be a finite abelian extension of $\Q$, and let ${}_LG = \operatorname{Gal}(L/\Q)$ be its
Galois group.  Class field theory provides a canonical continuous surjection
\[
\chi_L : \A_f^\times \longrightarrow {}_LG
\]
which is trivial on $\Q_+^\times$.  Restricting to the maximal compact subgroup
$\Zhat^\times \subset \A_f^\times$, we obtain the reciprocity homomorphism
\[
\chi_L : \Zhat^\times \longrightarrow {}_LG .
\]

\begin{defn}
\label{coverdef}
Let $L$ be a finite abelian extension of $\Q$ defined by the character $\chi_L$.
The associated geometric cover $\gamma(L)$ of $X_\Q$ is the projection
\[
\pi_{\chi_L} : X_\Q^{\chi_L} \longrightarrow X_\Q ,
\qquad \gamma(L) := \pi_{\chi_L},
\]
where the total space is the balanced product
\[
X_\Q^{\chi_L} := Y_\Q \times_{\Zhat^\times} {}_LG
               = (Y_\Q \times {}_LG)/\!\sim .
\]
The equivalence relation is induced by the diagonal action of $\Zhat^\times$:
for $u \in \Zhat^\times$, $y \in Y_\Q$, and $g \in {}_LG$,
\[
(yu, g) \sim (y, \chi_L(u) g).
\]
The map $\pi_{\chi_L}$ is induced by the projection
$Y_\Q \to Y_\Q/\Zhat^\times = X_\Q$.
\end{defn}

This construction associates to the  $\Zhat^\times$--bundle
$Y_\Q \to X_\Q$ the corresponding  ${}_LG$--bundle via the class field
homomorphism $\chi_L$.

The following theorem summarizes the fundamental properties of these covers,
establishing a precise correspondence between arithmetic data of the extension
$L/\Q$ and the topological and dynamical features of the space $X_\Q^{\chi_L}$ (singularities, monodromy, orbit decomposition).

\begin{thm}[{\cite[Theorem~3.12]{CC3}}]
\label{main}
Let $L$ be a finite abelian extension of $\Q$ with Galois group ${}_LG$.
\begin{enumerate}
  \item[(i)] \emph{Functoriality.}
  The assignment $L \mapsto \gamma(L)$ defines a contravariant functor
  \[
  \gamma : \mathbf{AbExt}(\Q) \longrightarrow \mathbf{Cov}_{\emph{ab}}(X_\Q)
  \]
  from finite abelian extensions of $\Q$ to finite abelian covers of $X_\Q$.

  \item[(ii)] \emph{Ramification.}
  The cover $\gamma(L)$ is unramified except at a finite set of places $R$,
  consisting of the archimedean place $\infty$ together with the finite primes
  that ramify in the extension $L/\Q$.

  \item[(iii)] \emph{Monodromy.}
  If $p \notin R$, the fiber over the periodic orbit
  $C_p \subset X_\Q$ is a disjoint union of circles.
  The monodromy along $C_p$ is given by the arithmetic Frobenius:
  \[
  \operatorname{Mon}(C_p) = \operatorname{Frob}_p \in {}_LG .
  \]

  \item[(iv)] \emph{Decomposition.}
  The connected components of $\pi_{\chi_L}^{-1}(C_p)$ are covering circles of $C_p$,
  in canonical bijection with the prime ideals $\mathfrak{p}$ of $L$ lying above $p$.
  The covering degree of each component equals the residue degree
  $f(\mathfrak{p}/p)$.
\end{enumerate}
\end{thm}

\subsection{Group structure of the fiber over a finite prime}\label{fibers}

We now refine the description of the universal cover by analyzing the fiber over a
finite periodic orbit $C_p$.  While this fiber was previously described topologically
as a mapping torus, we show here that the multiplicative structure of the adele class
space $Y_\Q$ endows it with a natural group structure.

Let $\pi : Y_\Q \to X_\Q$ be the canonical projection.  The fiber
\[
\mathcal{F}_p := \pi^{-1}(C_p)
\]
consists of adele classes $[a]\in Y_\Q$ such that the associated lattice $L_a$ (cf. Lemma~\ref{lem:spectrum}) is isomorphic
to $\Z[1/p]$.  Equivalently,
\[
a_p = 0, \qquad a_v \in \Z_v^\times \text{ for all } v \neq p,\infty,
\]
with $a_\infty \in \R^\times$ arbitrary.

\begin{prop}
The fiber $\mathcal{F}_p$ is a topological group under the multiplication of adele classes.
\end{prop}

\begin{proof}
Let $e_p \in \A$ be the idempotent adele defined by $(e_p)_p = 0$ and $(e_p)_v = 1$ for $v \neq p$. Its class $[e_p]$ belongs to $\mathcal{F}_p$.
For any $[a], [b] \in \mathcal{F}_p$, we have $a_p = b_p = 0$. Thus $(ab)_p = 0$.
For $v \neq p$, $a_v, b_v$ are units, so $a_v b_v$ is a unit.
Thus $[ab] \in \mathcal{F}_p$.
The element $[e_p]$ acts as the identity: for any $[a] \in \mathcal{F}_p$, $a e_p = a$ (since $0 \cdot 0 = 0$ and $x \cdot 1 = x$).
Inverses exist because the components $a_v$ for $v \neq p$ are units in $\Z_v$, and $a_\infty$ is invertible in $\R$.
\end{proof}

\begin{thm}
There is a canonical isomorphism of topological groups
\[
\mathcal{F}_p \cong C_\Q / \iota_p(\Q_p^\times),
\]
where $\iota_p: \Q_p^\times \hookrightarrow C_\Q$ is induced by the natural inclusion $\Q_p^\times \to \A^\times$ mapping $x$ to the idele with $x$ at $p$ and $1$ elsewhere.
\end{thm}

\begin{proof}
Multiplication by $e_p$ defines a surjective homomorphism
$\psi : \A^\times \to \mathcal{F}_p$, $\psi(u) = [u \cdot e_p]$.
Its kernel consists of ideles $u$ such that $[u e_p] = [e_p]$, i.e., $u e_p = q e_p$ for some $q \in \Q^\times$. The condition $u e_p = q e_p$ means:
\begin{itemize}
    \item[] At $v \neq p$: $u_v \cdot 1 = q \cdot 1 \implies u_v = q$.
    \item[] At $p$: $u_p \cdot 0 = q \cdot 0$ (always true).
\end{itemize}
Thus, $u$ must be equal to the rational number $q$ at all places except possibly $p$.
This means $u = q \cdot (1, \dots, 1, u_p/q, 1, \dots)$.
In terms of classes in $C_\Q = \A^\times / \Q^\times$, the kernel is precisely the image of the local group $\Q_p^\times$.
Thus $\psi$ induces an isomorphism $C_\Q / \Q_p^\times \xrightarrow{\sim} \mathcal{F}_p$.
\end{proof}

\begin{rem}
This description recovers the mapping--torus picture:
the quotient $C_\Q / \Q_p^\times$ is the group of idele classes away from $p$ modulo powers of $p$. 
\end{rem}

\subsection{The fiber over the archimedean orbit}

We next analyze the fiber over the archimedean orbit $C_\infty$.
The base point corresponds to the reduced element $\Z_\infty$, represented by the
idempotent adele $e_\infty=(1_f,0_\infty)$.

\begin{prop}
The fiber $\mathcal{F}_\infty := \pi^{-1}(C_\infty)$ is a compact topological group.
\end{prop}

\begin{proof}
The elements of the fiber $\mathcal{F}_\infty$ are adele classes $[a]$ such that $a_\infty = 0$ and $L_a = \Z$. The condition $L_a = \Z$ implies that the finite part $a_f$ of $a$ is a unit in $\Zhat$.
The multiplication is inherited from that on $\A$. The class $[e_\infty]$ is the neutral element. Since $\Zhat^\times$ is a group, every element is invertible.
\end{proof}

\begin{thm}
There are canonical isomorphisms
\[
\mathcal{F}_\infty \cong C_\Q/\R_+^\times \cong  \Zhat^\times \cong \operatorname{Gal}(\Q^{\text{ab}}/\Q).
\]
\end{thm}

\begin{proof}
Consider the map $\psi: \A^\times \to \mathcal{F}_\infty$ defined by $u \mapsto [u e_\infty]$.
This is a surjective homomorphism.
The kernel consists of ideles $u$ such that $u e_\infty = q e_\infty$ for some $q \in \Q^\times$.
This condition implies $u_f = q \cdot 1_f$, meaning $u$ is rational at all finite places.
Thus $u$ lies in the image of $\R^\times$ (embedded as $(1, \dots, 1, x)$ modulo $\Q^\times$),
and $C_\Q \cong \A_f^\times \times \R_+^\times / \Q_+^\times$.
The map  $\psi$ kills the $\R_+^\times$ component and the $\Q_+^\times$ action (since $q e_\infty = e_\infty$ in the class).
The result is thus isomorphic to $\A_f^\times / \Q_+^\times \cong \Zhat^\times$.
\end{proof}

\begin{rem}
This result completes the geometric picture of class field theory. 
\begin{itemize}[leftmargin=*, labelindent=0pt]
    \item[-] Over a finite prime $p$, the fiber $\mathcal F_p$ carries the information inherent to local class field theory of $p$ (modulo global units).
    \item[-] Over the infinite prime, the fiber $\mathcal F_\infty$ encodes the structure of the global Galois group $\operatorname{Gal}(\Q^{\text{ab}}/\Q)$.
\end{itemize}
The "singularity" at infinity thus contains the full arithmetic complexity of the abelian extensions of $\Q$.
\end{rem}

\subsection{The fiber over the generic point}

Finally, we consider the fiber $\mathcal{F}_\eta := \pi^{-1}([\Z])$ over the generic point $\Theta(\eta)=[\Z]$ of the
Jacobian, corresponding to a principal lattice with nonzero metric.

Precisely, 
a class $[a] \in Y_\Q$ if and only if its image in the  Jacobian is the pair $(\Z, 0)$ (representing the group $\Z$ and a non-zero metric).
This means:
\begin{enumerate}
    \item $a_\infty \neq 0$ (finite type).
    \item $L_a \cong \Z$ as groups (the underlying lattice is principal).
\end{enumerate}

\begin{thm}
The fiber over the generic point is canonically isomorphic to the idele class group
\[
\mathcal{F}_\eta \cong C_\Q.
\]
\end{thm}

\begin{proof}
The condition $a_\infty \neq 0$ is satisfied by all ideles.
The condition $L_a \cong \Z$ means that the fractional ideal generated by the finite part $a_f$ is a principal ideal. Since $\Z$ is a principal ideal domain (class number 1), every rank-1 subgroup of $\Q$ is isomorphic to $\Z$.
Therefore, the condition $L_a \cong \Z$ is satisfied for \textit{all} ideles $a \in \A^\times$.
Conversely, if $[a]$ is in the fiber, then $L_a \cong \Z$ implies $L_a = q \Z$ for some $q$, so $(q^{-1}a)_f \in \Zhat^\times$. Combined with $a_\infty \neq 0$, this implies $a$ is an idele.
Thus, the fiber is exactly the set of idele classes.
\end{proof}

\begin{rem}
The fibers thus assemble into the following  coherent picture:
\begin{itemize}[leftmargin=*, labelindent=0pt]
  \item the generic fiber $\mathcal F_\eta$ recovers the idele class group $C_\Q$;
  \item over the closed points $p$ and $\infty$ of $\spzb$, we find the "ramified" fibers (mapping tori) which are quotients of $C_\Q$ by the local units.
\end{itemize}
\end{rem}

The overall construction adds a top layer to the previous tower \eqref{2tow}, yielding a space $\widetilde{\spz}$ that covers $\mathcal{X}_\Q$
\begin{equation}\label{3floorb}
\begin{tikzcd}
\widetilde{\spz} \arrow[r,"\tilde\Theta"]\arrow[d] & Y_\Q\arrow[d]
\\
\mathcal{X}_\Q \arrow[r] \arrow[d] & X_\Q \arrow[d] \\
\spzb \arrow[r, "\Theta"] & \Jac(\spzb)
\end{tikzcd}
\end{equation}


\section{The geometric  monoids $\MFr(\spzb)$ and $\MRt(\spzb)$}\label{sectgeompic}

To establish a geometric interpretation of the adelic space $Y_\Q = \Q^\times \backslash \A$ we identify its points with arithmetic divisors equipped with \emph{rigidifying data} defined in terms of finite and infinite frames. We first introduce the notion of \emph{framed arithmetic divisors},  as torsion free rank-1 groups $L$ endowed with trivializations of their completions at both finite and infinite rational places. This formulation treats the archimedean and non-archimedean data symmetrically and realizes the adelic product as a tensor product. Subsequently, by applying Pontryagin duality to the finite frame, we derive the notion of \emph{rooted arithmetic divisors}. This dual perspective interprets the profinite frame as a coherent system of roots of unity, offering a combinatorial description that characterizes the divisor through its interaction with the torsion of the dual group.

\subsection{Framed arithmetic divisors}

\begin{defn}
A \emph{framed arithmetic divisor} is a triple $(L, \xi, \tau)$ consisting of the following data:
\begin{enumerate}
    \item A torsion-free rank-1 abelian group $L$.
    \item A \emph{finite frame} defined as a group homomorphism $$\xi: L \to \widehat{\Z}$$ generating the module $\operatorname{Hom}(L, \widehat{\Z})$.
    \item An \emph{infinite frame}, namely a homomorphism $$\tau: L \to \R$$ of abelian groups.
\end{enumerate}
\end{defn}

\begin{rem}
The infinite frame $\tau$ replaces the datum of a norm $\|\cdot\|$ on $L\otimes_\Z\R$. Indeed:
\begin{itemize}[leftmargin=*, labelindent=0pt]
    \item[] The norm corresponds to the absolute value: $\|x\| = |\tau(x)|$.
    \item[] The map $\tau$ retains the \emph{sign}  information (orientation) that corresponds to the connected components of the idele class group $C_\Q$.
    \item[] $\tau = 0$ corresponds to the singular fiber at infinity (\ie $a_\infty = 0$).
\end{itemize}
\end{rem}

The next lemma provides the adelic interpretation to a finite frame.

\begin{lem}\label{tight}
Let $L\subset \Q$ be a rank-1 group and $\xi \in \operatorname{Hom}(L, \widehat{\Z})$.\vspace{.03in}

$(i)$~There exists a unique finite rational adele $a_f\in \A_f$ such that 
$$
\xi(x)=a_f x, \ \ \forall x\in L.
$$
$(ii)$~The following statements are equivalent:
\begin{enumerate}
    \item $\xi$ is a generator of the $\widehat{\Z}$-module $\operatorname{Hom}(L, \widehat{\Z})$.
    \item The following equality of groups holds: $L=\{q\in \Q\mid a_f q\in \widehat{\Z}\}$.
\end{enumerate}
\end{lem}

\begin{proof}$(i)$~The tensor product decomposition $\A_f = \Q \otimes_{\Z} \Zhat$ and the isomorphism $\Q\otimes_{\Z} L\simeq \Q$ show that $\xi$ extends uniquely to a group homomorphism $$\tilde \xi:\Q\to \A_f.$$ Since $\A_f$ is a $\Q$-vector space, this determines uniquely an element $a_f\in \A_f$ such that $\tilde\xi(x)=a_f x, \ \ \forall x\in  \Q$.\newline
$(ii)$~By $(i)$ a homomorphism $\xi: L \to \widehat{\Z}$ is determined by a sequence of local multipliers $\alpha_p \in \Q_p$ such that $\xi(x)_p = \alpha_p x$.
Let $M = \operatorname{Hom}(L, \widehat{\Z})$. We analyze the local structure $M_p = \operatorname{Hom}(L, \Z_p)$.

\emph{The structure of the module $M_p$.}
The condition for $\alpha_p$ to define a map into $\Z_p$ can be stated as follows:
\[
\forall x \in L, \quad \alpha_p x \in \Z_p \iff v_p(\alpha_p) + v_p(x) \ge 0.
\]
In terms  of the arithmetic divisor $\sum n_p[p]$ of $L$, the minimal valuation $v_p(x)$ for $x\in L$ is $-n_p$, this requires:
\[
v_p(\alpha_p) - n_p \ge 0 \implies v_p(\alpha_p) \ge n_p.
\]
Two cases are possible:
\begin{itemize}[leftmargin=*, labelindent=0pt]
    \item[] If $n_p < \infty$, then $M_p = \{ \alpha_p \in \Q_p \mid v_p(\alpha_p) \ge n_p \} = p^{n_p} \Z_p$.
    \item[] If $n_p = \infty$, then $M_p = \{0\}$ (since no non-zero $\alpha_p$ maps $\Q_p$ into $\Z_p$).
\end{itemize}

\emph{The generator condition.}
$\xi$ generates $M$ as a $\widehat{\Z}$-module if and only if for every $p$, $\alpha_p$ generates $M_p$  as a module over $\Z_p$. Again, two cases are possible:
\begin{itemize}[leftmargin=*, labelindent=0pt]
    \item[] If $n_p < \infty$: $\alpha_p$ generates $p^{n_p} \Z_p$ if and only if $v_p(\alpha_p) = n_p$.
    \item[] If $n_p = \infty$: $0$ generates $\{0\}$. This is always true.
\end{itemize}
Thus, the statement 1. is equivalent to: $v_p(\alpha_p) = n_p$ for all $p$.\vspace{.04in}

\emph{Statement 2.}
Let $L'=\{q\in \Q\mid a_fq\in \widehat{\Z}\}$. It is  defined by:
\[
L' =   \{ x \in \Q \mid v_p(x) \ge -v_p(\alpha_p), \ \forall p \}.
\]
Statement 2. requires $L = L'$, which means the equality of the arithmetic divisors
$
 n_p = v_p(\alpha_p), \forall p
$.
Thus,  both statements $1$. and $2$. are equivalent to the equality of valuations $v_p(\alpha_p) = n_p$ for all  $p$.
\end{proof}

\subsection{The moduli space $\MFr(\spzb)$}

The multiplicative monoid of framed arithmetic divisors modulo isomorphism provides a first geometric interpretation of the adele class space $Y_\Q$.

\begin{defn}\label{moduli} The moduli space $\MFr(\spzb)$   is the set of framed arithmetic divisors modulo isomorphism. An isomorphism  $(L, \xi, \tau)\cong (L', \xi', \tau')$ is a group isomorphism $\phi: L \xrightarrow{\sim} L'$ such that $\xi' \circ \widehat{\phi} = \xi$ and $\tau' \circ \phi = \tau$.
\end{defn}

Given a rational adele $a = (a_f, a_\infty) \in \A$, one defines the map  $$\mathcal{F}: \A \longrightarrow \{ \text{Framed arithmetic divisors} \}\qquad \mathcal{F}(a) = (L_a, \xi_a, \tau_a)$$
where:
\begin{itemize}[leftmargin=*, labelindent=0pt]
    \item[-] $L_a = \{ q \in \Q \mid q a_f \in \widehat{\Z} \}$.
    \item[-] $\xi_a: L_a \to \widehat{\Z}$ is defined by $\xi_a(x) = x a_f$.
    \item[-] $\tau_a: L_a \to \R$ is given by $\tau_a(x) = x a_\infty$.
\end{itemize}

\begin{thm}
The map $\mathcal{F}$ induces a canonical bijection between  $Y_\Q = \Q^\times \backslash \A$ and the moduli space $\MFr(\spzb)$ of framed arithmetic divisors.
\end{thm}

\begin{proof}
\emph{Descent to Quotient.} 
Let $q \in \Q^\times$. Consider the adele $a' = qa$.
The associated group is $L_{qa} = q^{-1} L_a$.
The map $\phi: L_a \to L_{qa}$ given by multiplication by $q^{-1}$ is an isomorphism of groups.
\begin{itemize}[leftmargin=*, labelindent=0pt]
    \item[] Finite Frame: $\xi_{qa}(\phi(x)) = \xi_{qa}(q^{-1}x) = (q^{-1}x)(qa)_f = x a_f = \xi_a(x)$.
    \item[] Infinite Frame: $\tau_{qa}(\phi(x)) = \tau_{qa}(q^{-1}x) = (q^{-1}x)(qa)_\infty = x a_\infty = \tau_a(x)$.
\end{itemize}
Thus, the triples are isomorphic, and $\mathcal F$ descends to $Y_\Q$.\vspace{.03in}

We check the \emph{Bijectivity.}
Let $(L, \xi, \tau)$ be a framed arithmetic divisor.
Since $L$ is rank-1 torsion-free group, we can choose an embedding $\iota: L \hookrightarrow \Q$.

 By Lemma \ref{tight} $(i)$, the map $\xi$ extends to a map $\Af \to \Af$, which is multiplication by some $a_f \in \Af$. By Lemma \ref{tight} $(ii)$  the   generator condition ensures that $L$ is recovered correctly.
 
The map $\tau$ extends to a map $\Q \otimes \R \to \R$, which is multiplication by some $a_\infty \in \R$.

This yields an adele $a = (a_f, a_\infty)$.
Changing the embedding $\iota$ by a factor $q \in \Q^\times$ changes $a$ to $q^{-1}a$, which preserves the class in $Y_\Q$.
Thus, the correspondence is one-to-one.
\end{proof}

\subsubsection{The monoid structure}

The set of framed arithmetic divisors carries a natural tensor product structure.
Given $D_1 = (L_1, \xi_1, \tau_1)$ and $D_2 = (L_2, \xi_2, \tau_2)$, we define $D_1 \otimes D_2 = (L, \xi, \tau)$ by:
\begin{itemize}[leftmargin=*, labelindent=0pt]
    \item[] $L = L_1 \otimes_\Z L_2$.
    \item[] $\xi = \xi_1 \otimes \xi_2$ (product in $\widehat{\Z}$).
    \item[] $\tau = \tau_1 \otimes \tau_2$ (product in $\R$).
\end{itemize}
Explicitly, for the infinite part one has: $\tau(x \otimes y) = \tau_1(x) \tau_2(y)$.

\begin{thm}
The sets bijection $Y_\Q \cong \MFr(\spzb)$ is an isomorphism of monoids.
If $a, b \in \A$ correspond to arithmetic framed divisors $D_a, D_b$, then the product $ab$ corresponds to the tensor product of arithmetic divisors $D_a \otimes D_b$.
\end{thm}

\begin{proof}
Let $a, b \in \A$.
The product adele is $ab = (a_f b_f, a_\infty b_\infty)$.
The divisor for the product is defined by:
\begin{itemize}[leftmargin=*, labelindent=0pt]
    \item[] $\xi_{ab}(x) = x (a_f b_f)$.
    \item[] $\tau_{ab}(x) = x (a_\infty b_\infty)$.
\end{itemize}
Consider the tensor product $D_a \otimes D_b$.
The map $L_a \otimes L_b \to L_{ab}$ given by multiplication in $\Q$ is an isomorphism.
Under this identification:
\begin{itemize}[leftmargin=*, labelindent=0pt]
    \item[] The finite frame acts as $\xi_a(x) \xi_b(y) = (x a_f)(y b_f) = (xy) (a_f b_f)$, matching $\xi_{ab}$.
    \item[] The infinite frame acts as $\tau_a(x) \tau_b(y) = (x a_\infty)(y b_\infty) = (xy) (a_\infty b_\infty)$, matching $\tau_{ab}$.
\end{itemize}
Thus, the structures coincide.
\end{proof}
\subsection{Rooted arithmetic divisors and the moduli $\MRt(\spzb)$}

 Given a torsion free rank one group $L$ we let $L^\vee:= \Hom(L,\Q/\Z)$ be its dual (see Appendix \ref{appendual}). Given   a frame $\xi \in \operatorname{Hom}(L, \widehat{\Z})$ the dual frame   $\xi^\vee\in \Hom(\Q/\Z,L^\vee)$ specifies a coherent system of roots of unity, more precisely of torsion elements belonging to $\rm{Tor}(L^\vee)$. 
We introduce the notion of a \emph{rooted arithmetic divisor} to  reflect this fact.

\subsubsection{The Support of a divisor}

\begin{defn}
Let $L$ be a torsion free rank one group. Let $S$ be the set of primes where $L$ is divisible (\ie the \emph{singular} primes). 
We define the subgroup of supported roots of unity as follows:
\[
\mu_{(S)} := \bigoplus_{p \notin S} \Q_p/\Z_p \subset \Q/\Z.
\]
Equivalently, $\mu_{(S)}$ is the subgroup of roots of unity whose order is prime to all $p \in S$.
\end{defn}

\subsubsection{The isomorphism condition}

\begin{defn}
A root on a torsion free rank one group $L$ is a homomorphism
\[
\rho: \Q/\Z \longrightarrow \rm{Tor}(L^\vee)
\]
satisfying the following two conditions:
\begin{enumerate}
    \item \emph{Vanishing on singularities:} $\rho$ vanishes on the singular components:
    \[
    \ker(\rho) \supseteq \bigoplus_{p \in S} \Q_p/\Z_p.
    \]
    \item \emph{Isomorphism on support:} The restriction of $\rho$ to the supported subgroup induces an isomorphism:
    \[
    \rho|_{\mu_{(S)}}: \mu_{(S)} \xrightarrow{\sim} \rm{Tor}(L^\vee).
    \]
\end{enumerate}
\end{defn}
Let $\xi \in \operatorname{Hom}(L, \widehat{\Z})$. The Pontryagin dual $\xi^*$ of $\xi$ is a group homomorphism from the Pontryagin dual $\Q/\Z$ of $\widehat{\Z}$ to the Pontryagin dual $L^*$ of $L$. By Proposition \ref{dualities} of Appendix \ref{appendual}, the  range of $\xi^*$ is contained in $L^\vee:= \Hom(L,\Q/\Z)$, thus this defines the dual map 
\begin{equation}\label{dualmap}
\xi^\vee\in \Hom(\Q/\Z,L^\vee), \ \ \xi^\vee(r)=\xi^*(r), \ \forall r\in \Q/\Z.
\end{equation}
\begin{prop}\label{isodual}
Let $\xi \in \operatorname{Hom}(L, \widehat{\Z})$. The following statements are equivalent:
\begin{enumerate}
    \item $\xi$ is a generator of the $\widehat\Z$-module (\ie  $\xi$ is a frame).
    \item The dual map $\xi^\vee$ satisfies the isomorphism condition 2. above.
\end{enumerate}
\end{prop}

\begin{proof}
Let $L\subset \Q$, $D=(n_p)$ its associated divisor, and $S$  the set of all primes with $n_p=\infty$. Let $a=(a_p)\in \A_f$ be a finite adele such that $L=\{q\in \Q\mid aq\in \Zhat\}$. By Proposition \ref{proptor} 
the torsion subgroup of the dual $L^\vee$ is the subgroup :
\[
\rm{Tor}(L^\vee) = \bigoplus_{p \notin S} \Q_p/a_p\Z_p \subset L^\vee
\]
By Lemma \ref{tight}, the map $\xi$ is given locally at a prime $p \notin S$ by multiplication by a non-zero element  $\alpha_p\in \Q_p$ and condition 1. is equivalent to $v_p(\alpha_p)=n_p=v_p(a_p)$ for all $p\notin S$. 
The map $\xi^\vee_p: \Q_p/\Z_p \to \Q_p/a_p\Z_p$ is the multiplication by  $\alpha_p\in \Q_p$.
Thus, $\xi^\vee_p$ is an isomorphism if and only if $v_p(\alpha_p) = v_p(a_p)$.
This is precisely the condition for $\xi_p$ to be a generator of $\operatorname{Hom}(L, \Z_p)$.
\end{proof}

\begin{defn}
A \emph{rooted arithmetic divisor} is a triple $(L, \rho, \tau)$ consisting of a torsion free rank-1 group $L$ equipped with a root $\rho$ and a morphism $\tau:L\to \R$.
\end{defn}

The term rooted associated to an arithmetic divisor  refers to the action of the map $\rho$ on the standard roots of unity. 
Let $\zeta_n = e^{2\pi i / n} \in \Q/\Z$ be a primitive $n$-th root of unity.
The map $\rho$ assigns to $\zeta_n$ a character $\chi_n \in \operatorname{Hom}(L, \Q/\Z)$.
This character can be viewed as an arithmetic root of unity on the group $L$, where
\begin{itemize}[leftmargin=*, labelindent=0pt]
    \item[] For $L=\Z$, $\rho$ assigns a specific root of unity on the circle to the abstract generator of the dual.
    \item[] Over a periodic orbit (\ie for $L=\Z[1/p]$), $\rho$ assigns roots for orders prime to $p$, effectively rooting the divisor at the unramified places.
\end{itemize}

In other words, a rooted arithmetic divisor is an arithmetic divisor where the phase ambiguity (Galois ambiguity) has been resolved by fixing a specific configuration of roots of unity.\vspace{.08in}

Let $a = (a_f, a_\infty) \in \A$ be a rational adele. We associate to it the triple 
\begin{equation}\label{tau}
\sigma(a)=(L_a, \rho_a, \tau_a),
\end{equation}
where:
\begin{enumerate}
    \item $L_a = \{ q \in \Q \mid a_f q \in \widehat{\Z} \}$.
    \item The morphism $\tau_a$ is defined as $\tau_a(q) = a_\infty q\in \R$.
    \item $\rho_a=\xi_a^\vee$ where $\xi_a: L_a\to \widehat{\Z}$,    $\xi_a(x) = x a_f$.
\end{enumerate}
Proposition \ref{isodual} shows that the duality functor $D(\xi):=\xi^\vee$ gives a bijective correspondence between framed divisors and rooted divisors. We now investigate the meaning of the product in terms of rooted divisors.

\subsection{Tensor product of roots via character multiplication}

We formalize the tensor product of rooted divisors using the multiplicative structure of the finite rings $\Z/m\Z$. This approach naturally handles both the composition of frames and the vanishing on singular sets.

\subsubsection{The character pairing}

Let $L_1, L_2$ be two torsion-free abelian groups. Fix an integer $m \ge 1$.
The ring structure of $\Z/m\Z$ induces a canonical bilinear map:
\[
\Phi_m: \operatorname{Hom}(L_1, \Z/m\Z) \times \operatorname{Hom}(L_2, \Z/m\Z) \longrightarrow \operatorname{Hom}(L_1 \otimes L_2, \Z/m\Z)
\]
defined by the pointwise product of values:
\[
\Phi_m(\chi_1, \chi_2)(x \otimes y) := \chi_1(x) \cdot \chi_2(y) \pmod m.
\]
(This is well-defined because the map $(u,v) \mapsto uv$ is bilinear over $\Z/m\Z$).

\subsubsection{Product of roots}

Recall that the torsion subgroup of the dual of a torsion-free abelian group $L$  is the colimit of the finite levels:
\[
\rm{Tor}(L^\vee) \cong \varinjlim_m \operatorname{Hom}(L, \Z/m\Z).
\]
A root $\rho$ is determined by the family of characters $\rho(\frac{1}{m}) \in \operatorname{Hom}(L, \Z/m\Z)$. This family is consistent with the relations in the group $\Q/\Z$.
Specifically, for any integers $n$ and $k$, the element $1/n$ in $\Q/\Z$ is equal to $k \cdot (1/nk)$.
Therefore, any root $\rho$ must satisfy the consistency condition:
\begin{equation}\label{consist}
\rho(1/n) = k \cdot \rho(1/nk).
\end{equation}

\begin{prop}
Let $(L_1, \rho_1)$ and $(L_2, \rho_2)$ be two rooted divisors.
The following equation defines a  root $\rho$ for the tensor product $L_1 \otimes L_2$  by the pairing of the characters at each level $m$:
\begin{equation}\label{tensor}
\rho\left(\frac{1}{m}\right) = \Phi_m\left( \rho_1\left(\frac{1}{m}\right), \, \rho_2\left(\frac{1}{m}\right) \right).
\end{equation}
Explicitly, for any $x \in L_1, y \in L_2$:
\[
\left\langle \rho\left(\frac{1}{m}\right), x \otimes y \right\rangle = \left\langle \rho_1\left(\frac{1}{m}\right), x \right\rangle \cdot \left\langle \rho_2\left(\frac{1}{m}\right), y \right\rangle \in \Z/m\Z.
\]
\end{prop}

\begin{proof} We verify that the definition \eqref{tensor} of the product root fulfills the consistency condition \eqref{consist}. 
We check that if $\rho_1$ and $\rho_2$ satisfy this condition, their product $\rho$ (defined level-wise) also satisfies it.
Let $N = nk$.
We represent the values of the roots by elements in the finite rings:
Let $u_m \in \Z/m\Z$ represent the value $\rho_1(1/m)$ (via the identification $\operatorname{Hom}(L, \Z/m\Z) \cong \Z/m\Z \otimes L^\vee \dots$ effectively treating them as scalars for this check).
The condition $\rho_1(1/n) = k \rho_1(1/N)$ translates to the congruence of representatives:
\[
\frac{u_n}{n} \equiv k \frac{u_N}{N} \pmod 1 \implies \frac{u_n}{n} \equiv \frac{u_N}{n} \pmod 1 \implies u_n \equiv u_N \pmod n.
\]
Thus, the sequence of values $(u_m)_m$ forms an element of the projective limit $\varprojlim (\Z/m\Z) \cong \widehat{\Z}$.
Let $(u_m)$ and $(v_m)$ be the sequences for $\rho_1$ and $\rho_2$.
We have:
\[
u_n \equiv u_N \pmod n \quad \text{and} \quad v_n \equiv v_N \pmod n.
\]
The product root $\rho$ is defined by the product in the ring $\Z/m\Z$:
\[
w_m := u_m \cdot v_m \pmod m.
\]
We must verify that $w_n \equiv w_N \pmod n$.
Compute the reduction of $w_N$ modulo $n$:
\[
w_N \pmod n = (u_N \cdot v_N) \pmod n = (u_N \pmod n) \cdot (v_N \pmod n).
\]
Using the hypothesis:
\[
=u_n \cdot v_n = w_n.
\]
Thus, the consistency condition holds.
\end{proof}

\begin{rem}
This calculation reveals that the transition maps relating the values at level $N$ to level $n$ are simply the canonical projections $\pi: \Z/N\Z \to \Z/n\Z$. Since these projections are ring homomorphisms, they preserve the multiplicative structure defined at each level.
\end{rem}

\subsubsection{Automatic handling of singularities}

Next formula automatically reproduces the filtering behavior of the singular sets.\vspace{.05in}

Let $p$ be a singular prime for $L_1$ (\ie $p \in S_1$).
Then $L_1$ is $p$-divisible, so for any power $q=p^k$, there are no non-zero homomorphisms from $L_1$ to $\Z/q\Z$:
\[
\operatorname{Hom}(L_1, \Z/p^k\Z) = 0 \implies \rho_1\left(\frac{1}{p^k}\right) = 0.
\]
Consequently, the product character is zero:
\[
\rho\left(\frac{1}{p^k}\right) = 0 \cdot \rho_2\left(\frac{1}{p^k}\right) = 0.
\]
This confirms that if $p \in S_1 \cup S_2$, then the product root $\rho$ vanishes on the $p$-power roots of unity.
Thus, the support of the product root is exactly the intersection of the supports:
\[
\operatorname{Supp}(\rho) = \operatorname{Supp}(\rho_1) \cap \operatorname{Supp}(\rho_2).
\]
\begin{rem}[The hidden coproduct structure]
The definition of the product root $\rho = \rho_1 \otimes \rho_2$ relies on the formula $\rho(1/n) = \rho_1(1/n) \cdot \rho_2(1/n)$, suggesting the existence of a diagonal coproduct $\Delta$ on $\Q/\Z$ satisfying $\Delta(1/n) = 1/n \otimes 1/n$.

While such a map is ill-defined in the category of abelian groups (where $\Q/\Z \otimes_\Z \Q/\Z = 0$), it finds a rigorous formulation in the category of Ind-rings.
We view $\Q/\Z$ as the direct limit of the finite rings $R_n = \Z/n\Z$. For each $n$, the multiplication map $\mu_n: R_n \otimes_\Z R_n \to R_n$ is an isomorphism of rings. Its inverse provides a diagonal map:
\[
\Delta_n: R_n \xrightarrow{\sim} R_n \otimes_\Z R_n, \quad \text{satisfying }~ \Delta_n(1_n) = 1_n \otimes 1_n.
\]
Identifying the element $1/n \in \Q/\Z$ with the unity $1_n \in R_n$, the tensor product of roots can be interpreted as the composition:
\[
\rho = (\rho_1 \otimes \rho_2) \circ \Delta,
\]
where $\Delta = \varinjlim \Delta_n$ is the limit of these diagonal maps. This structural perspective explains why the product is anchored at the generators $1/n$ (the multiplicative identities) and why it scales non-linearly for general elements $k/n$.
\end{rem}

\subsection{Compatibility of duality with tensor products}

We prove that the duality functor  $D(\xi) = \xi^\vee$ transforms the product of frames into the tensor product of roots.

\begin{lem}\label{dualtensor}
Let $(L_1, \xi_1)$ and $(L_2, \xi_2)$ be two framed arithmetic divisors.
Let $\xi = \xi_1 \otimes \xi_2$ be the product frame on $L_1 \otimes L_2$.
Let $\rho_1 = \xi_1^\vee$ and $\rho_2 = \xi_2^\vee$ be the associated roots.
Then the root associated to the product frame is the tensor product of the roots:
\[
(\xi_1 \otimes \xi_2)^\vee = \xi_1^\vee \otimes \xi_2^\vee.
\]
\end{lem}

\begin{proof}
It suffices to verify the equality of the maps on the generators $1/n \in \Q/\Z$ acting on pure tensors $x \otimes y \in L_1 \otimes L_2$.\vspace{.03in}

\emph{The dual of the product frame (LHS).}
Let $\rho = \xi^\vee$. By the definition of the dual frame, the pairing is given by the reduction of the frame value modulo $n$:
\[
\langle \rho(1/n), x \otimes y \rangle = \xi(x \otimes y) \pmod n.
\]
By the definition of the product frame, $\xi(x \otimes y)$ is the product in the ring $\widehat{\Z}$:
\[
\xi(x \otimes y) = \xi_1(x) \cdot \xi_2(y).
\]
Let $\pi_n: \widehat{\Z} \to \Z/n\Z$ be the canonical projection. Since $\pi_n$ is a ring homomorphism:
\[
\langle \rho(1/n), x \otimes y \rangle = \pi_n\left( \xi_1(x) \cdot \xi_2(y) \right) = \pi_n(\xi_1(x)) \cdot \pi_n(\xi_2(y)).
\]

\emph{The product of the roots (RHS).}
Let $\tilde{\rho} = \rho_1 \otimes \rho_2$. By the definition of the tensor product of roots (using the ring structure of $\Z/n\Z$) one has:
\[
\langle \tilde{\rho}(1/n), x \otimes y \rangle = \langle \rho_1(1/n), x \rangle \cdot \langle \rho_2(1/n), y \rangle \quad (\text{product in } \Z/n\Z).
\]
Using the definition of the individual roots $\rho_i = \xi_i^\vee$:
\[
\langle \rho_1(1/n), x \rangle = \xi_1(x) \pmod n = \pi_n(\xi_1(x)),
\]
\[
\langle \rho_2(1/n), y \rangle = \xi_2(y) \pmod n = \pi_n(\xi_2(y)).
\]
Substituting these back:
\[
\langle \tilde{\rho}(1/n), x \otimes y \rangle = \pi_n(\xi_1(x)) \cdot \pi_n(\xi_2(y)).
\]
Comparing the two expressions, we see they are identical for all $n$ and all pure tensors.
\[
\langle (\xi_1 \otimes \xi_2)^\vee(1/n), x \otimes y \rangle = \langle (\xi_1^\vee \otimes \xi_2^\vee)(1/n), x \otimes y \rangle.
\]
Thus, the duality map respects the tensor structure.
\end{proof}

\subsection{Canonical monoid isomorphism}

\begin{thm}\label{unif}
The duality functor $D(\xi) = \xi^\vee$ induces a canonical isomorphism of monoids between moduli of framed and rooted arithmetic divisors:
\begin{equation}\label{iso}
D: \MFr(\spzb) \xrightarrow{\sim} \MRt(\spzb), \quad (L, \xi, \tau) \longmapsto (L, \xi^\vee, \tau).
\end{equation}
This map respects the tensor structure, where the product of roots is defined via the diagonal action on the system of finite rings $\Z/n\Z$.

The map $\sigma$ in \eqref{tau} induces  a canonical isomorphism of the adelic monoid  $Y_\Q$ with the the rooted Picard monoid $\MRt(\spzb)$.
\end{thm}
\begin{proof}
By Proposition \ref{isodual},  the duality functor $D(\xi) = \xi^\vee$ gives the isomorphism \eqref{iso}.  By  Lemma \ref{dualtensor} the duality functor  preserves the monoid operation. 
\end{proof}

The above duality isomorphism $D$ justifies the introduction of the following definition:
\begin{equation}\label{}
\tspz := \MFr(\spzb) \cong \MRt(\spzb)
\end{equation}


\section{$\tspz$ and  spectral realization of zeros of $L$-functions}\label{Wtrace}

In this section we re-interpret the diagram \eqref{3floorb} as the key pullback diagram  
\begin{equation}\label{}
\begin{tikzcd}
\widetilde{\spz} \arrow[r,"\tilde\Theta"]\arrow[d,"\tilde\pi"'] & \tspz\arrow[d,"\pi"]
 \\
\spzb \arrow[r, "\Theta"] & \Jac(\spzb)
\end{tikzcd}
\end{equation} 
to explain the relationship between the geometric theory of  $\tspz$ and the spectral realization of the zeros of L-functions. We show that the trace formula of \cite{Co-zeta} admits a geometric reinterpretation as a Lefschetz trace formula for the action of the idèle class group on $\tspz$, viewed as \emph{supporting the cohomology} of the arithmetic curve.
Specifically, we show that the geometric side of Weil's explicit formula is supported on the image of the universal Abel-Jacobi (embedding) map $\widetilde{\spz} \stackrel{\tilde\Theta}{\hookrightarrow} \tspz$, with the local terms arising directly from the dynamics on the transverse space to the fibers of this embedding.\vspace{.03in}

At this point, it is essential to stress a fundamental difference between the approach pursued in this paper and the classical geometric theory of function fields.
In the function field case, the Picard group is a discrete abelian group, and the spectral information is extracted from the action of the Frobenius endomorphism on the $\ell$-torsion points of the Jacobian. The eigenvalues of the Frobenius acting on these finite dimensional vector spaces (transferred from $\Q_\ell$ to $\C$) are the zeros of the zeta function. The fixed points of iterates of the Frobenius on the curve correspond to prime divisors, yielding the zeta function via the Lefschetz trace formula.

In the number field case, on the other hand,  we work with the Picard monoid $\tspz$ and consider the action of the idèle class group $C_\Q$ by \emph{translations}.
If we were dealing with a  group (like the classical Picard group), translations would have no fixed points (except for the identity), rendering a trace formula trivial.
The key insight of this approach is instead that the adelic space $\tspz$ is not discrete  but a geometric space with rich internal structure. The action of $C_\Q$ by translation on this space is non-trivial and possesses  fixed points located on the range of the Abel-Jacobi map.

\subsection{The explicit formula as a trace formula}\label{sect8.1}
Noncommutative geometry offers powerful tools for analyzing the rich geometry of the adelic spaces $Y_{\mathbf{Q}}$ and $X_{\mathbf{Q}}$. This framework is especially compelling because it furnishes the analytic machinery required for the spectral realization of L-functions.

Let $\alpha$ be a non-trivial character of the additive group $\A / \Q$, and  $\alpha = \prod_v \alpha_v$ its decomposition into local factors. Let $h \in \mathcal{S}(C_\Q)$ be a Schwartz function with compact support  on the idele class group $C_\Q$.
The generalized Weil explicit formula relates the zeros of $L$-functions with Grössencharakter to a sum of local integrals. It takes the following form:
\begin{equation}\label{weilexplicit}
\hat{h}(0) + \hat{h}(1) - \sum_{\chi \in \widehat{C_{\Q, 1}}} \sum_{\rho \in Z_{\tilde{\chi}}} \hat{h}(\tilde{\chi}, \rho) = \sum_v \int_{\Q_v^\times}^{\prime} \frac{h(u^{-1})}{|1-u|} d^* u.
\end{equation}
Here:
\begin{itemize}
    \item The sum on the left runs over the non-trivial zeros $\rho$ of $L$-functions associated to characters $\chi$ of the compact group $C_{\Q, 1} = C_\Q / \R_+^\times$.
    \item $\hat{h}(\chi, z) = \int_{C_\Q} h(u) \chi(u) |u|^z d^*u$ is the Mellin transform.
    \item The integral $\int'$ on the right in \eqref{weilexplicit} is normalized by the local factors $\alpha_v$ following Tate's thesis.
\end{itemize}

Weil's explicit formula can be interpreted as the trace formula:
\[
\operatorname{Tr}_{\text{distr}} \left( \int_{C_\Q} h(u) \vartheta(u) d^*u \right) = \sum_v \int_{\Q_v^\times} \frac{h(u^{-1})}{|1-u|} d^* u.
\]
Here, $\vartheta(u)$ denotes the scaling action of an element $u \in C_\Q$ on the space of functions on the adelic space $Y_\Q = \A / \Q^\times$, defined by $$\vartheta(u) \xi(x) = \xi(u^{-1} x).$$

\subsection{Geometric interpretation: isotropy and fibers}\label{geomtrace}

The geometric side of the trace formula (the right-hand side of \eqref{weilexplicit}) is a sum over the places $v$ of $\Q$, which correspond to the closed points $x \in \spzb$.
In this geometric framework, these contributions are understood as localized on the range of the Abel-Jacobi map
$$
\widetilde{ \spz}\stackrel{\tilde\Theta}{\longrightarrow} \tspz.
$$
This range is the union 
 of the fibers $\pi^{-1}(x)$ of the projection $\pi: \tspz \to \Jac(\spzb)$, and each fiber contributes to the geometric side of the trace formula by $\int_{\Q_v^\times} \frac{h(u^{-1})}{|1-u|} d^* u$. 
For each place $v$, the local multiplicative group $\Q_v^\times$ embeds as a subgroup of the idele class group $\Q_v^\times \subset C_\Q$.
Geometrically, these subgroups appear as the isotropy groups of the action of $C_\Q$ on  $\tspz$. 

The contribution of a place $v$ to the trace formula arises from the action of the isotropy group $\Q_v^\times$ on the transverse space to the orbit at that point. This transverse space is identified with the local field $\Q_v$.
The appearance of the term $\frac{1}{|1-u|}$ can be justified by a formal computation of the distributional trace of the scaling operator on the local transverse space.

Consider the scaling operator $T = \vartheta(u^{-1})$ acting on functions on the line $\Q_v$ (or $\R$):
\[
T \xi(x) = \xi(u x).
\]
We can represent this operator as an integral operator with a distributional kernel $k(x, y)$:
\[
T \xi(x) = \int k(x, y) \xi(y) dy, \qquad \text{where }~ k(x, y) = \delta(u x - y).
\]
The distributional trace is then obtained by integrating the kernel along the diagonal $x=y$:
\[
\operatorname{Tr}_{\text{distr}}(T) = \int k(x, x) dx = \int \delta(u x - x) dx = \int \delta((u-1)x) dx.
\]
Using the scaling property of the delta distribution $\delta(\lambda x) = \frac{1}{|\lambda|} \delta(x)$, we obtain:
\[
\operatorname{Tr}_{\text{distr}}(T) = \frac{1}{|u-1|} \int \delta(x) dx = \frac{1}{|u-1|}.
\]
This calculation demonstrates that the geometric side of the explicit formula is a sum of local traces associated to the hyperbolic flow on the fibers over the arithmetic  Jacobian.

\subsection{The semilocal trace formula}\label{sectsemi}
In order to obtain a geometric proof of \eqref{weilexplicit}, one first notices that the functional analysis focuses on the locally compact rings  
\[
\A_S := \prod_{v \in S} \Q_v
\]
which are {\em finite} products of the local fields $\Q_v$ over a finite set $S$ of places containing the archimedean place $\infty$. Indeed, 
the function space associated with the adele class space $ Y_\mathbf{Q}$  is represented as a cross product and consists of functions defined on the adeles, specifically belonging to the Bruhat–Schwartz space. These functions are characterized by their non-trivial dependence on only finitely many places. In other words one  replaces the infinite product involved in the construction of adeles by a finite product of the local fields $\Q_v$. All the techniques and tools available in the global case  have a natural  analogue in this semilocal framework. The semilocal version of the group $\Q^\times$ is the group $\Z_S^\times$ of invertible elements in the ring $\Z_S$  of $S$-integers, where each prime $p\in S$ is inverted.  One obtains  the semilocal analogues of the global Bruhat–Schwartz algebra $\cS(\A)$ and  its cross product with $\Q^\times$,
\begin{equation}\label{semiloc1} 
\cS(\A_{S}), \  \  \cS(\A_{S})\rtimes \Z_S^\times.
\end{equation}
The key advantage of working in the semilocal framework is that for a finite product of local fields the Haar measures of the additive and multiplicative groups are comparable: this  fails at the global level. One thus obtains a natural Hilbert space $L^2( \Z_S^\times\backslash \A_{S})$  of $0$-forms (that admits a Hochschild homological interpretation) and 
then investigates the action $\theta$  of the semilocal idele class group $C_{\Q,S}=\Z_S^\times\backslash \A_{S}^\times$ on  $L^2( \Z_S^\times\backslash \A_{S})$:
    \begin{equation}\label{scalingS}
(\theta(u) \, \xi) (a) = \xi (u^{-1} a), \ \ \  \forall
a \in  \Z_S^\times\backslash \A_{S}, \ \forall \xi\in L^2( \Z_S^\times\backslash \A_{S}).
\end{equation}
The local to global inclusion of the multiplicative groups  $\Q_v^\times\subset C_\Q$ of the local fields $\Q_v$ in the idele class group continues to make sense in the semilocal case and gives canonical inclusions $\Q^\times_v\subset C_{\Q,S}$ for $v\in S$.

To prove the finiteness of the trace, one  performs a cutoff in phase space, \ie by limiting both $p,q$ to absolute value $\leq \lambda$.  To be more precise, let $P_\lambda$ be the projection, in the Hilbert space $L^2( \Z_S^\times\backslash \A_{S})$ on the subspace of functions which vanish for  $\vert a\vert >\lambda$, and $\widehat{P_\lambda}$ its conjugate by the Fourier transform.  The cutoff  is achieved by the operator $R_\lambda = \widehat P_\lambda  P_\lambda$ where $ P_\lambda$ effects the cutoff $\vert q\vert \leq \lambda$ and $\widehat P_\lambda$ effects the cutoff $\vert p\vert \leq \lambda$. The semilocal trace formula expresses the trace of the cutoff scaling action as follows. Let $h$ be a test function with compact support on $C_{\Q,S}$ and $\theta(h):=\int h(u)\theta(u)d^*u$, then:
\begin{equation}\label{semilocTrace}
\Tr (\theta(h)R_\lambda) = 2h (1) \log \lambda + \sum_{v \in S}
\int'_{\Q^\times_v} \frac{h(u^{-1}) }{ \vert 1-u \vert} \, d^* u + o(1).
\end{equation}
This is the original formulation of the semilocal trace formula in \cite{Co-zeta}.  The divergent term $2h (1) \log \lambda$ signals the presence of the white light coming from the (trace of the) regular representation. In the perspective of this paper this is the contribution of the image of the generic point $\eta\in \spzb$ by the Abel-Jacobi map. 

\subsection{Sheaves over $\spz$}\label{sectshespec}
The semilocal algebras $\mathcal{S}\left(\mathbb{A}_S\right) \rtimes \mathbf{Z}_S^{\times}$ are constructed as the cross products of the Bruhat-Schwartz algebras $\mathcal{S}(\mathbb{A}_S)$ of functions on semilocal adeles, by the multiplicative groups $\mathbf{Z}_S^{\times}=\mathbf{G}_m(S^c)$. These are groups of sections of the sheaf $\mathbf{G}_m$ on the Zariski open set complement of the closed subset $S\subset \Spec \mathbf{Z}$.
It is a remarkable fact that the semilocal algebras \eqref{semiloc1} form a sheaf of algebras over $\spz$.  One thus obtains two sheaves ${\mathscr{O} }$ and  ${\mathscr{O} }\rtimes \mathbf{G}_m$ on $\spz$ whose sections, for any finite set of places $S\ni \infty$, are
\begin{equation}\label{semiloc2} 
{\mathscr{O} }(S^c)=\cS(\A_{S}), \  \  \left({\mathscr{O} }\rtimes \mathbf{G}_m\right)(S^c)= \cS(\A_{S})\rtimes \Z_S^\times.
\end{equation} 
These sections  encapsulate fully the semilocal framework, and one finds that the global algebras are simply given by the stalk of these sheaves at the generic point   $\eta\in \spz$.

We recall from \cite{CC3} the following result which establishes the compatibility of the noncommutative geometric constructions with the algebraic geometry of $\spzb$.

\begin{thm}
(1) The algebraic cross product $\mathscr{O} \rtimes \mathbf{G}_m$ defines a sheaf of algebras on $\spz$ such that for every finite set of places $S \ni \infty$

$$
\left(\mathscr{O} \rtimes \mathbf{G}_m\right)\left(S^c\right)=\mathcal{S}\left(\mathbb{A}_S\right) \rtimes \mathbf{Z}_S^{\times}.
$$

(2) The stalk of $\mathscr{O} \rtimes \mathbf{G}_m$ at the generic point $\eta\in \spz$ is the global cross product $\mathcal{S}\left(\mathbb{A}_{\mathbf{Q}}\right) \rtimes \mathbf{Q}^{\times}$.

(3) The global sections of $\mathscr{O} \rtimes \mathbf{G}_m$ on $\spz$ form the cross product $\mathcal{S}(\mathbf{R}) \rtimes\{ \pm 1\}$.
\end{thm}



 \appendix
 \section{Adelic description of the dual  of torsion free rank-1 groups}\label{appendual}

\subsection{Finite adeles and $\Hom(\Q, \Q/\Z)$}
 Let $\A_f = \Q \otimes_{\Z} \Zhat$ be the ring of finite adeles and 
 $\Zhat = \prod_p \Z_p$ its maximal compact subring.
We identify the quotient group $\A_f / \Zhat$ with $\Q/\Z$ via the canonical projection:
\begin{equation} \label{quotad}
\pi: \A_f \longrightarrow \A_f/\Zhat \xrightarrow{\sim} \Q/\Z.
\end{equation}
Explicitly, if $a = (a_p)_p \in \A_f$, then $\pi(a) = \sum_p \{a_p\}_p \pmod 1$, where $\{x\}_p$ denotes the `fractional part' of the $p$-adic number $x$. 
Let 
\begin{equation} \label{quotad1}
\psi: \A_f \to \Hom(\Q, \Q/\Z)\quad 
\psi(a)(q) = \pi(qa) \quad \text{for } a \in \A_f,~ q \in \Q.
\end{equation}
\begin{lem}\label{isopsi}
The map $\psi$ is an isomorphism of topological groups.
\end{lem}

\begin{proof}
We proceed by establishing injectivity and surjectivity using the structural properties of $\A_f$ and $\Q$.

\emph{Injectivity.} 
Suppose $\psi(a) = 0$ for some $a \in \A_f$. By definition, this means:
$
\pi(qa) = 0~ \forall q \in \Q
$. 
This implies that $qa \in \ker(\pi) = \Zhat~ \forall q \in \Q$. The following cases arise:
\begin{itemize}[leftmargin=*, labelindent=0pt]
    \item[] If $q=1$, then  $a \in \Zhat$.
    \item[] If $q=1/n$ for an integer $n$, then $a/n \in \Zhat$, which implies $a \in n\Zhat$.
\end{itemize}
Thus, we conclude that $a \in \bigcap_{n \in \Z \setminus \{0\}} n\Zhat$.
Since $\Zhat \cong \prod_p \Z_p$ and $\bigcap_k p^k \Z_p = \{0\}$, it follows that $a = 0$.

\emph{Surjectivity.} 
We compare the defining exact sequence of $\A_f$ 
\begin{equation}\label{seqA}
0 \longrightarrow \Zhat \longrightarrow \A_f \xrightarrow{\pi} \Q/\Z \longrightarrow 0
\end{equation}
with the exact sequence obtained by applying the functor $\Hom(-, \Q/\Z)$ to the defining sequence of $\Q$ 
\[0 \to \Z \to \Q \to \Q/\Z \to 0.
\]
Since $\Q/\Z$ is divisible (injective in the category of abelian groups), the functor $\Hom(-, \Q/\Z)$ is exact, thus we have
\begin{equation}\label{seqB}
0 \longrightarrow \Hom(\Q/\Z, \Q/\Z) \longrightarrow \Hom(\Q, \Q/\Z) \longrightarrow \Hom(\Z, \Q/\Z) \longrightarrow 0
\end{equation}
Let us analyze the terms in \eqref{seqB}:
\begin{itemize}[leftmargin=*, labelindent=0pt]
    \item[] $\Hom(\Z, \Q/\Z) \cong \Q/\Z$ (determined by the image of 1).
    \item[] $\Hom(\Q/\Z, \Q/\Z) \cong \Zhat$. This is the ring of endomorphisms of the torsion group $\Q/\Z$. An element $u \in \Zhat$ acts by multiplication.
\end{itemize}
Thus, \eqref{seqB} becomes:
\begin{equation}\label{seqB1}
0 \longrightarrow \Zhat \longrightarrow \Hom(\Q, \Q/\Z) \longrightarrow \Q/\Z \longrightarrow 0
\end{equation}

The map $\psi$ in \eqref{quotad1} induces a morphism between \eqref{seqA} and \eqref{seqB1}:
\[
\begin{tikzcd}
0 \arrow[r] & \Zhat \arrow[r] \arrow[d, "\psi|_{\Zhat}"] & \A_f \arrow[r, "\pi"] \arrow[d, "\psi"] & \Q/\Z \arrow[r] \arrow[d, "\text{id}"] & 0 \\
0 \arrow[r] & \Hom(\Q/\Z, \Q/\Z) \arrow[r] & \Hom(\Q, \Q/\Z) \arrow[r] & \Hom(\Z, \Q/\Z) \arrow[r] & 0
\end{tikzcd}
\]
The right vertical arrow is the identity (canonical isomorphism): for $a \in \A_f$, restricting $\psi(a)$ to $\Z$ gives $n \mapsto \pi(na)$, which is determined by $\pi(a) \in \Q/\Z$.
 
 The left vertical arrow maps $u \in \Zhat$ to the homomorphism $x \mapsto \pi(ux)$ on $\Q/\Z$. This is the canonical identification of $\Zhat$ with $\operatorname{End}(\Q/\Z)$.

By the Short Five Lemma, since the outer vertical maps are isomorphisms, the middle map $\psi$ is an isomorphism.
Next, we consider the investigation of $\psi$ in topological terms.\vspace{.02in}

\emph{Topology.}
$\A_f$ carries the restricted product topology, making it a locally compact group, while 
$\Hom(\Q, \Q/\Z)$ carries the compact-open topology. Since $\Q$ is discrete, this is the topology of pointwise convergence.
The isomorphism $\psi$ is a homeomorphism because it matches the compact open subgroups, namely:\vspace{.02in}

- The maximal compact subgroup of $\A_f$ is $\Zhat$.

- The maximal compact subgroup of $\Hom(\Q, \Q/\Z)$ is $\Hom(\Q/\Z, \Q/\Z) \cong \Zhat$ (those maps vanishing on $\Z$).

The filtration of $\A_f$ by $n^{-1}\Zhat$ corresponds exactly to the filtration of homomorphisms by the order of their restriction to $\Z$.
\end{proof}

\subsection{The dual of a rank-1 group}

We derive an explicit formula for the character dual $L^\vee$ of a rank-1 subgroup $L \subset \Q$ using the adelic annihilator.
\begin{lem}
Let $L \subseteq \mathbf{Q}$ be a subgroup. Its character dual
\[
L^{\vee} = \Hom(L,\mathbf{Q}/\mathbf{Z})
\]
identifies canonically with the quotient of the finite adèle ring by the adelic annihilator of L,
\[
L^{\vee} \;\cong\; \mathbb{A}_f / L^{\perp}.
\]
Here the adelic annihilator $L^{\perp}$ is defined by
 \[
 L^{\perp}
\;=\;
\{\, a \in \mathbb{A}_f \mid aL \subseteq \widehat{\mathbf{Z}} \,\}
\;=\;
\{\, a \in \mathbb{A}_f \mid \pi(ha)=0 \text{ for all } h\in L \,\}.
\]
\end{lem}

\begin{proof}
Consider the short exact sequence of discrete groups:
\[
0 \longrightarrow L \xrightarrow{i} \Q \longrightarrow \Q/L \longrightarrow 0.
\]
Since $\Q/\Z$ is an injective $\Z$-module (it is divisible), the functor $\Hom(-, \Q/\Z)$ is exact. Applying this functor yields the exact sequence of dual groups:
\[
0 \longleftarrow \Hom(L, \Q/\Z) \xleftarrow{i^*} \Hom(\Q, \Q/\Z) \xleftarrow{\pi^*} \Hom(\Q/L, \Q/\Z) \longleftarrow 0
\]
Using the identification $\Hom(\Q, \Q/\Z) \cong \A_f$, we can identify the term $\Hom(\Q/L, \Q/\Z)$ with the subgroup of $\A_f$ consisting of elements that vanish on $L$ (under the pairing).
\[
\ker(i^*) = \text{Image}(\pi^*) \cong \{ a \in \A_f \mid \psi(a)(h) = 0, \, \forall h \in L \}.
\]
The condition $\psi(a)(h) = 0$ in $\Q/\Z$ is equivalent to $ah \in \Zhat$ in $\A_f$. Thus, $\ker(i^*) = L^\perp$.
By the First Isomorphism Theorem:
\[
L^\vee = \text{Image}(i^*) \cong \A_f / \ker(i^*) = \A_f / L^\perp
\]
\end{proof}

Let $S$ be the set of primes $p$ such that $L$ is $p$-divisible (i.e., $L \otimes \Z_p \cong \Q_p$).
This set $S$ characterizes the "singularities" of the group $L$. Precisely, we have:
\begin{itemize}[leftmargin=*, labelindent=0pt]
    \item[] For $L = \Z$, $S = \emptyset$.
    \item[] For $L = \Z[1/p]$, $S = \{p\}$.
    \item[] For $L = \Q$, $S = \mathcal{P}$ (the set of all primes).
\end{itemize}
We call the restricted product 
\[
\A_S = {\prod_{p \in S}}^{'} \Q_p
\]
the \emph{ring of $S$-adeles}.

\begin{prop} Let $a\in \A_f$ and $L=\{q\in \Q\mid aq\in \Zhat\}$. Let $S$ be the set of all primes with $a_p=0$. 
The adelic annihilator of $L$ is given by:
\[
L^\perp = \left( \prod_{p \notin S}a_p \Z_p \right) \times \{0\}_S \subset \A_f.
\]
The dual group decomposes as:
\begin{equation}\label{dualgr}
L^\vee \cong \A_f / L^\perp \cong \left({\prod_{p \notin S}}^{'} \Q_p/a_p\Z_p\right) \times \A_S.
\end{equation}
\end{prop}

\begin{proof}
The condition $x \in L^\perp$ means $x h \in \widehat{\Z}$ for all $h \in L$. This condition factorizes locally as: $x_p h_p \in \Z_p$ for all $h_p \in L_p$ where $L_p:=L\otimes_\Z \Z_p$. One sees that:
\begin{itemize}[leftmargin=*, labelindent=0pt]
    \item[] If $p \notin S$, then $L_p \cong a_p^{-1} \Z_p$. The condition $x_p L_p \subset \Z_p$ is equivalent to  $x_p \in a_p\Z_p$. Thus the local component is $a_p\Z_p$.
    \item[] If $p \in S$, then $L_p \cong \Q_p$. The condition $x_p \Q_p \subset \Z_p$ implies $x_p = 0$ (since $\Q_p$ is unbounded). Thus the local component is $\{0\}$.
\end{itemize}
The quotient $\A_f / L^\perp$ is  the restricted product of its components, where:
\begin{itemize}[leftmargin=*, labelindent=0pt]
\item[] For $p \in S$: $\Q_p / \{0\} \cong \Q_p$. The product of these gives $\A_S$.
    \item[] For $p \notin S$: $\Q_p /a_p \Z_p$. The restricted product of these gives $\prod'_{p \notin S} \Q_p/a_p\Z_p$.  
\end{itemize}
\end{proof}
The restricted product $\prod'_{p \notin S} \Q_p/a_p\Z_p$ contains the torsion group $\bigoplus_{p \notin S} \Q_p/a_p\Z_p$. The group $\Q_p/a_p\Z_p$ is isomorphic, by multiplication by $a_p$, to the $p$-torsion group $\Q_p/\Z_p=\mu_{p^\infty}$.

\begin{rem}It is not true in general that the restricted product $\prod'_{p \notin S} \Q_p/a_p\Z_p$ is equal to the torsion group $\bigoplus_{p \notin S} \Q_p/a_p\Z_p$. In fact,  this fails as soon as there are infinitely many $p\notin S$ such that  
$n_p=v(a_p)>0$. 
Here $\Z_p / a_p \Z_p$ is non-trivial for infinitely many $p$. The restricted product allows for elements that have non-zero entries in these finite groups for infinitely many primes. Consequently, $\prod'_{p \notin S} \Q_p/a_p\Z_p$ is uncountable and contains elements of infinite order (e.g.  a sequence $(x_p)$ where $x_p \in \Z_p/a_p\Z_p$ has order $p^{v(a_p)}$, and these orders grow to infinity).
\end{rem}
\subsection{The torsion subgroup of the dual}

\begin{prop}
Let $L$ be a torsion-free abelian group. There is a canonical isomorphism:
\[
\rm{Tor}( \operatorname{Hom}(L, \Q/\Z)) \cong \varinjlim_{n} \operatorname{Hom}(L, \Z/n\Z),
\]
where the limit is taken over the directed system of integers ordered by divisibility ($n \mid m$), with transition maps induced by the canonical inclusions $\Z/n\Z \hookrightarrow \Z/m\Z$ (mapping $1 \mapsto m/n$).
\end{prop}

\begin{proof}
Recall that $\Q/\Z$ is the directed union of its finite subgroups:
\[
\Q/\Z = \bigcup_{n \ge 1} \left( \frac{1}{n}\Z \big/ \Z \right) \cong \varinjlim_{n} \Z/n\Z.
\]
Let $\phi \in \operatorname{Hom}(L, \Q/\Z)$.
The element $\phi$ belongs to the torsion subgroup if and only if there exists an integer $n \ge 1$ such that $n \cdot \phi = 0$.
The condition $n \cdot \phi = 0$ means that for all $x \in L$, $n \phi(x) = 0$ in $\Q/\Z$.
This is equivalent to saying that the image of $\phi$ lies in the subgroup of elements of order dividing $n$:
\[
\operatorname{Im}(\phi) \subseteq \frac{1}{n}\Z \big/ \Z \cong \Z/n\Z.
\]
Thus, a homomorphism is torsion if and only if it factors through one of the finite groups $\Z/n\Z$.
The colimit $\varinjlim_{n} \operatorname{Hom}(L, \Z/n\Z)$ consists precisely of equivalence classes of such maps, which maps isomorphically to the set of maps into $\Q/\Z$ with bounded image.
\end{proof}
\begin{prop}\label{proptor} Let $a\in \A_f$ and $L=\{q\in \Q\mid aq\in \Zhat\}$. Let $S$ be the set of all primes with $a_p=0$. 
The torsion subgroup of the dual $L^\vee$ is the subgroup :
\[
\rm{Tor}(L^\vee) = \bigoplus_{p \notin S} \Q_p/a_p\Z_p \subset L^\vee
\]
It is isomorphic to the group $\mu_{(S)}$ of roots of unity of order prime to $S$.
\end{prop}
\proof By \eqref{dualgr} the dual group is the product of the two terms
$$
L^\vee \cong  \left({\prod_{p \notin S}}^{'} \Q_p/a_p\Z_p\right) \times \A_S.
$$
Let $x\in \rm{Tor}(L^\vee)$ be such that $nx=0$. Since $\A_S$ is torsion free one has
$$
x\in \left({\prod_{p \notin S}}^{'} \Q_p/a_p\Z_p\right)
$$
and for primes $q$ not dividing $n$ the component $x_q$ is $x_q=0$ since $\Q_q/a_q\Z_q$ is isomorphic to $\Q_q/\Z_q$ whose $n$ torsion is $0$. Thus $x$ is a finite sum of elements of $\Q_p/a_p\Z_p$ for primes dividing $n$.\endproof
\endproof 
\subsection{Consistency with Pontryagin duality}

We clarify the relationship between the character dual $L^\vee = \operatorname{Hom}(L, \Q/\Z)$ and the standard Pontryagin dual $L^* = \operatorname{Hom}(L, \R/\Z)$.

\begin{prop}\label{dualities}
Let $H$ be a discrete abelian group.
\begin{enumerate}
    \item There is a strict inclusion $\rm{Tor}(H^*) \subseteq H^\vee \subset H^*$.
    \item Any homomorphism from a torsion group into the Pontryagin dual factors through the character dual. Specifically:
    \[
    \operatorname{Hom}(\Q/\Z, H^*) \cong \operatorname{Hom}(\Q/\Z, H^\vee).
    \]
    \item Furthermore, the image of any such map lies in the torsion subgroup $\rm{Tor}(H^\vee)$.
\end{enumerate}
\end{prop}

\begin{proof}
1. Let $\chi \in H^*$ be a torsion element. Then $n\chi = 0$ for some $n$, implying $\chi(H) \subseteq \frac{1}{n}\Z/\Z \subset \Q/\Z$. Thus $\chi \in H^\vee$.
However, $H^\vee$ may contain elements of infinite order (e.g., the $\Q_p$ component when $H=\Z[1/p]$), so $L^\vee$ is generally larger than $\rm{Tor}(H^*)$.

2. Let $\Phi: \Q/\Z \to H^*$. For any $z \in \Q/\Z$, let $n$ be its order. Then $n \Phi(z) = \Phi(nz) = 0$. Thus $\Phi(z)$ is a torsion element in $H^*$. By (1), $\Phi(z) \in H^\vee$.

3. Since $\Phi(z)$ has finite order, it belongs to $\rm{Tor}(H^\vee)$. In the case $H=\Z[1/p]$, where $H^\vee \cong \rm{Tor}(H^\vee) \times \Q_p$, this implies the map is zero on the $\Q_p$ factor.
\end{proof}


\section{$\Pic(\spz)$ and the monoid of the points of $\nat$}\label{appB}

We analyze the categorical structure of the topos associated with the multiplicative monoid $\mathbf{N}^\times$, and we show that the arithmetic tensor product of rank-1 groups arises naturally from the multiplication of integers.

Let $M = \mathbf{N}^\times$ be the monoid of non-zero natural numbers under multiplication. We consider Grothendieck's presheaf topos $\nat = \mathbf{PSh}(M^{\mathrm{op}})$, that is the category of sets equipped with a right action of $M$.
A point of  $\nat$ is a geometric morphism $p: \mathbf{Set} \to \nat$. By Diaconescu's theorem, such points are classified by flat functors $F: M \to \mathbf{Set}$ (or equivalently, flat left $M$-sets).
Since $M$ is a commutative cancellative monoid, the category of flat $M$-sets is equivalent to the category of ordered  rank-1 abelian groups (viewed as direct limits of the free $M$-set). Thus, a point $p$ of $\nat$ is uniquely determined, up to isomorphism, by a torsion free rank $1$  group since the two orders on such a group are isomorphic by $x\mapsto -x$.

The multiplication of integers defines a monoid homomorphism:
$
\mu: M \times M \rightarrow M$, $(m, n) \mapsto mn$.

This map induces a geometric morphism between the corresponding toposes. On the level of points, this structure induces a product operation. Given two points $L_1$ and $L_2$ (viewed as flat $M$-sets), their external product $L_1 \boxtimes L_2$ is a flat $(M \times M)$-set. The product point $L$ is obtained by the left Kan extension along $\mu$:
\[
L = \operatorname{Lan}_\mu (L_1 \boxtimes L_2).
\]
Algebraically, this Kan extension corresponds to the tensor product over the monoid ring $\mathbf{Z}[M]$ (that acts on the sets via the integers): $
L \cong L_1 \otimes_{\mathbf{Z}[M]} L_2$.

Since the action of $M$ on these groups is compatible with the standard multiplication in $\mathbb{Z}$, this tensor product over the monoid coincides with the standard tensor product of abelian groups $
L \cong L_1 \otimes_{\mathbf{Z}} L_2$. 
Thus, the monoidal structure on the topos $\nat$ induced by the multiplication of integers corresponds precisely to the tensor product of the rank-1 groups parametrizing its points.

\vspace{5pt}
\small
\noindent

\begin{tabular}{p{0.48\textwidth} p{0.48\textwidth}}
\raggedright
Alain Connes\\
Coll\`ege de France\\
3 Rue d'Ulm\\
75005 Paris\\
France\\
\href{mailto:alain@connes.org}{alain@connes.org}
&
\raggedleft
Caterina Consani\\
Department of Mathematics\\
Johns Hopkins University\\
Baltimore, MD 21218\\
USA\\
\href{mailto:cconsan1@jhu.edu}{cconsan1@jhu.edu}
\end{tabular}


\begin{thebibliography}{9}

\bibitem{Baer} R.~Baer, \emph{ Abelian groups without elements of finite order}. Duke Mathematical Journal 3, no. 1 (1937): 68--122.

 
\bibitem{Co-zeta} 
A.~Connes, 
\emph{Trace formula in noncommutative geometry and the zeros of the Riemann zeta function}, 
Selecta Math. (N.S.) 5 (1999), no.~1, 29--106.

\bibitem{CC1} A.~Connes, C.~Consani, {\em Geometry of the Arithmetic Site}. Adv. Math. 291 (2016), 274--329.

\bibitem{CC2}  A. Connes, C. Consani, \emph{Absolute Algebra and Segal's $\Gamma$-Rings: au Dessous de $\overline{\Spec \mathbf{Z}}$}, J. Number Theory 162 (2016), 518-551.

\bibitem{CC3} A. Connes, C. Consani, \textit{Knots, primes and class field theory}, to appear in Regulators V proceedings, Contemporary Mathematics AMS., arXiv:2501.06560.


\bibitem{Mazur} B.~Mazur
\emph{ Notes on étale cohomology of number fields}. Ann. Sci. École Norm. Sup. (4) 6 (1973), 521-552 (1974). 

\bibitem{Ngo} B.C. Ng\^o, \emph{On a certain sum of automorphic L-functions}. In Automorphic Forms and
Related Geometry: Assessing the Legacy of I.I. Piatetski-Shapiro, volume 614 of Contemp. Math., pages 337-343. Amer. Math. Soc., Providence, RI, 2014.


\bibitem{Rotman} J.~J.~Rotman \emph{An introduction to the theory of groups} Fourth edition, Graduate texts in mathematics, 138, Springer.

\bibitem{Serre} J.~P.~Serre, \emph{Corps locaux}. (French) Publications de l'Institut de Mathématique de l'Université de Nancago, VIII. Actualités Scientifiques et Industrielles  No. 1296. Hermann, Paris, 1962. 

\bibitem{Vinberg}  E.~B.~Vinberg, \emph{On reductive algebraic semigroups}, in Lie groups and Lie algebras: E. B. Dynkin's Seminar, Amer. Math. Soc. Transl. Ser. 2, 169 (1995), 145--182.


\bibitem{Weil} 
A.~Weil, 
\emph{Sur les formules explicites de la th\'eorie des nombres premiers}, 
Comm. S\'em. Math. Univ. Lund, 1952, 252--265.
\end{thebibliography}
\end{document}